\documentclass{amsart}
\usepackage{amsmath}
\usepackage{amssymb}
\usepackage{amsthm}
\usepackage{dsfont}
\usepackage{lineno}
\usepackage{subfigure}
\usepackage{graphicx}
\usepackage{multirow}
\usepackage{color}
\usepackage{xspace}
\usepackage{pdfsync}
\usepackage{hyperref} 
\usepackage{algorithmicx}
\usepackage{algpseudocode}
\usepackage{algorithm}
\usepackage{mathtools}
\usepackage{enumitem}
\graphicspath{{Figures/}}

\DeclareMathOperator*{\argmin}{argmin}

\DeclarePairedDelimiter\floor{\lfloor}{\rfloor}
\newcommand{\bq}{\begin{equation}}
\newcommand{\eq}{\end{equation}}
\newcommand{\R}{\mathbb{R}}

\newcommand{\abs}[1]{\left\vert#1\right\vert}
\newcommand{\norm}[1]{\left\Vert#1\right\Vert}

\newcommand{\G}{\mathcal{G}}
\newcommand{\bO}{\mathcal{O}}

\newcommand{\Dt}{\mathcal{D}}

\newcommand{\Sf}{\mathbb{S}^{2}}
\newcommand{\Tf}{\mathcal{T}}
\newcommand{\Zf}{\mathcal{Z}}

\newcommand{\Nf}{\mathcal{N}}

\newcommand{\MA}{Monge-Amp\`ere\xspace}

\algnewcommand{\LineComment}[1]{\State \(\triangleright\) #1}

\newtheorem{theorem}{Theorem}

\theoremstyle{lemma}
\newtheorem{lemma}[theorem]{Lemma}

\newtheorem{corollary}[theorem]{Corollary}

\newtheorem{definition}[theorem]{Definition}

\newtheorem{remark}[theorem]{Remark}

\newtheorem{hypothesis}[theorem]{Hypothesis}

\theoremstyle{remark}

\makeatletter
\newcommand\appendix@section[1]{%
\refstepcounter{section}%
\orig@section*{Appendix \@Alph\c@section: #1}%
}
\let\orig@section\section
\g@addto@macro\appendix{\let\section\appendix@section}
\makeatother

\begin{document}

\title[A numerical method for optimal transport on the sphere]{A convergent finite difference method for optimal transport on the sphere}

\author{Brittany Froese Hamfeldt}
\address{Department of Mathematical Sciences, New Jersey Institute of Technology, University Heights, Newark, NJ 07102}
\email{bdfroese@njit.edu}
\author{Axel G. R. Turnquist}
\address{Department of Mathematical Sciences, New Jersey Institute of Technology, University Heights, Newark, NJ 07102}
\email{agt6@njit.edu}

\thanks{The first author was partially supported by NSF DMS-1619807 and NSF DMS-1751996. The second author was partially supported by  an NSF GRFP}

\begin{abstract}
We introduce a convergent finite difference method for solving the optimal transportation problem on the sphere. The method applies to both the traditional squared geodesic cost (arising in mesh generation) and a logarithmic cost (arising in the reflector antenna design problem).  At each point on the sphere, we replace the surface PDE with a Generated Jacobian equation posed on the local tangent plane using geodesic normal coordinates.  The discretization is inspired by recent monotone methods for the \MA equation, but requires significant adaptations in order to correctly handle the mix of gradient and Hessian terms appearing inside the nonlinear determinant operator, as well as the singular logarithmic cost function.  Numerical results demonstrate the success of this method on a wide range of challenging problems involving both the squared geodesic and the logarithmic cost functions.
\end{abstract}

\date{\today}    
\maketitle

We consider the problem of optimal transportation on the sphere.  That is, given two prescribed density functions $f_1$ and $f_2$, we seek a mapping $T:\Sf\to\Sf$ such that
\bq\label{eq:OT}
T = \argmin\limits_{T_\# f_1 = f_2} \int_{\Sf} c(x,T(x)) f_1(x) dS(x).
\eq
Here $c(x,y)$ is the cost of transporting a unit of mass from $x$ to $y$ and $T_\# f_1 = f_2$ indicates that
\[ \int_A f_1(x)\,dS(x) = \int_{T(A)} f_2(y)\,dS(y) \]
for every measurable $A \subset \Sf$.

In recent years, optimal transport has emerged as an important component in many different applications including image registration~\cite{HakerRegistration},
 astrophysics (estimating the shape of the early universe)~\cite{FrischUniv}, 
 meteorology~\cite{Cullen}, machine learning~\cite{Peyre_ComputationalOT}, and geophysical inversion problems~\cite{EF_Seismic}, among others.  This has led to the development of a number of different numerical methods and convergence results for solving the optimal transport problem in Euclidean geometries~\cite{BenamouBrenier,BenamouDuval,BFO_OTNum,BuddWilliams,HamfeldtBVP2,Rubinstein_OT, Prins_OT,Schmitzer_OT,Yadav_MA}.

A less well studied but equally important setting is optimal transport between density functions on the sphere.
Perhaps the simplest cost is the squared geodesic distance
\[ c(x,y) = \frac{1}{2}d_{\mathbb{S}^2}(x,y)^2, \]
where $d_{\Sf}(x,y)$ denotes the geodesic distance between $x,y\in\Sf$.  This cost function has recently been applied to the problem of mesh generation on the sphere in the context of meteorology~\cite{McRae_OTonSphere,Weller_OTonSphere}.

A second cost of particular interest is the log cost
\[ c(x,y) = -\log \norm{x-y}, \]
which arises in the reflector antenna design problem~\cite{GlimmOliker_SingleReflector,Wang_Reflector2}. The notation $\norm{\cdot}$ denotes the Euclidean distance in the ambient space $\mathbb{R}^3$.

Recently, some progress has been made in the numerical solution of the Optimal Transport problem on the sphere in the case of the squared geodesic cost.  The work of~\cite{Weller_OTonSphere} used a geometric interpretation of a \MA type equation on the sphere to produce the first such method.  A finite element solution of this \MA type equation was produced in~\cite{McRae_OTonSphere}. For problems posed on a subset of the sphere, the stereographic projection can be used to reframe the problem as an optimal transport problem on the plane (with non-quadratic cost); this was the approach of~\cite{RomijnSphere}. The authors of the present article recently introduced a simple framework for proving the convergence of numerical methods for optimal transport on the sphere~\cite{HT_OTonSphere}. However, to date we are not aware of any methods that fit within this framework.

In this article, we produce the first convergent PDE-based method for solving the Optimal Transport problem on the sphere.  This method is based on an approximation of a Generated Jacobian equation on local tangent planes, using carefully constructed local coordinates.  The discretization is inspired by recent monotone methods that have been proposed for the \MA equation.  However, a complicating factor is the presence of gradient terms mixed with the Hessian terms inside of a nonlinear operator.  This requires the introduction of new techniques for approximating both first- and second-order terms in order to preserve both the consistency and the monotonicity of our scheme.  Additionally, the logarithmic cost function requires a careful regularization in order to preserve the well-posedness of the method.  We produce an implementation and present computational results that demonstrate the success of this method for both the squared geodesic cost and the logarithmic cost. 

\section{Background}\label{sec:background}
\subsection{Optimal transport on the sphere}

We consider points $x,y$ lying on a unit sphere $\Sf$ centered at the origin.  We are interested in two different cost functions $c(x,y)$: the squared geodesic distance on the sphere,
\bq\label{eq:squareCost}
c(x,y) =  \frac{1}{2}d_{\Sf}(x,y)^2 = \frac{1}{2}\left(2\sin^{-1}\left(\frac{\norm{x-y}}{2}\right)\right)^2,
\eq
and the log-cost arising in the reflector antenna problem,
\bq\label{eq:logCost}
c(x,y) = -\log \norm{ x-y }.
\eq

The optimal map corresponding to each cost function is determined from the conditions
\vspace*{-4pt}
\bq\label{eq:mapConditionSphere}
\begin{cases}
\nabla_{\mathbb{S}^2,x}c \left( x,T(x,p) \right) = -p, & x\in\Sf, p\in \Tf_x\\
T(x,p) \in \Sf
\end{cases} \vspace*{-4pt}
\eq
%where $\Tf(x)$ denotes the tangent plane to the sphere at the point $x$.  
%As in~\cite{McRae_OTonSphere}, the map can be found explicitly as
%\vspace*{-4pt}
%\bq\label{eq:mapSphere}
%T(x,p) = \cos\left(\frac{\abs{p}}{2}\right) x + \sin\left(\frac{\abs{p}}{2}\right)\frac{p}{\abs{p}}. \vspace*{-4pt}
%\eq
where $\Tf_x$ denotes the tangent plane at $x$. These mappings can be found explicitly for the cost functions we are interested in.  In the case of the squared geodesic cost~\cite{McRae_OTonSphere}, the map is
\bq\label{eq:mapd2}
T(x,p) = \cos\left(\norm{ p } \right) x + \sin\left(\norm{ p } \right)\frac{p}{ \norm{ p }}. 
\eq
In the case of the logarithmic cost~\cite{HT_OTonSphere}, the map is
\bq\label{eq:mapLog}
T(x,p) = x\frac{\norm{ p }^2-1/4}{\norm{ p }^2+1/4}-\frac{p}{\norm{ p }^2+1/4}.
\eq

The solution to the optimal transport problem is then given by
\bq\label{eq:PDE1}
F \left( x,\nabla_{\mathbb{S}^2} u(x), D_{\mathbb{S}^2}^2u(x) \right) = 0
\eq
where
\bq\label{eq:OTPDE}
F(x,p,M) \equiv -\det \left( M+A(x,p) \right) + H(x,p)
%{\det}^+(D_S^2u(x) + A(x,\nabla_S u(x))) = F(x,\nabla_S u(x)) \vspace*{-4pt}
\eq
subject to the $c$-convexity (ellipticity) condition, which requires
\bq\label{eq:cconvex}
D_{\mathbb{S}^2}^2u(x) + A(x,\nabla_{\mathbb{S}^2} u(x)) \geq 0.
\eq
Here
\begin{align*}
A(x,p) &= D_{\mathbb{S}^2,xx}^2c \left( x,T(x,p) \right)\\
H(x,p) &= \abs{\det{D_{\mathbb{S}^2,xy}^2c \left( x,T(x,p) \right)}}f_1(x)/f_2 \left( T(x,p) \right),
\end{align*}
and the PDE now describes a nonlinear relationship between the surface gradient and Hessian on the sphere.

This PDE belongs to the class of Generated Jacobian equations~\cite{Trudinger_GJE}, and places constraints on the Jacobian of the mapping $T(x,\nabla u(x))$ in order to force the density $f_1$ to be transported into the density $f_2$.  It is similar to the \MA equation that is seen in optimal transport in Euclidean space (with quadratic cost $c(x,y) = \frac{1}{2}\|x-y\|^2$).  However, the introduction of more complicated geometries and cost functions now leads to a mix of gradient and Hessian terms inside of the nonlinear determinant operator.

This problem was studied by Loeper~\cite{Loeper_OTonSphere}, who showed that under mild regularity requirements on the data, the optimal transport problem on the sphere admits a smooth ($C^3$) solution $u$.  
\begin{hypothesis}[Conditions on data]\label{hyp:Smooth}
We require problem data to satisfy the following conditions:
\begin{enumerate}
\item[(a)] There exists some $m>0$ such that $f_2(x) \geq m$ for all $x\in\Sf$.
\item[(b)] The mass balance condition holds, $\int_{\Sf} f_1(x)\,dx = \int_{\Sf} f_2(y)\,dy$.
\item[(c)] The cost function is either $c(x,y) = \frac{1}{2}d_{\Sf}(x,y)^2$ or $c(x,y) = -\log \norm{ x-y }$.
\item[(d)] The data satisfies the regularity requirements $f_1, f_2 \in C^{1,1}(\Sf)$.
\end{enumerate}
\end{hypothesis}

Weak ($C^1$) solutions are also possible for discontinuous density functions that are only in $L^p$ ($1 \leq p < \infty$).  Moreover, the solution is unique up to additive constants.

%From Loeper~\cite{Loeper_OTonSphere} we obtain regularity results for the solution $u$, which apply to both cost functions of interest to us. Suppose that the density functions $f_1, f_2 \in C^{1,1}(\mathbb{S}^2)$ (resp. $C^{\infty}(\mathbb{S}^2)$) with $f_2$ bounded away from zero. Then $u \in C^{3, \alpha}(\mathbb{S}^2)$, for every $\alpha \in [0,1)$ (resp. $u \in C^{\infty}(\mathbb{S}^2)$).

\subsection{Approximation of elliptic PDEs}
The PDE operators we consider in this work are degenerate elliptic.
\bq\label{eq:PDEelliptic} F\left( x,u(x),\nabla u(x), D^2u(x) \right) = 0, \quad x\in\bar{\Sf}.\eq
\begin{definition}[Degenerate elliptic]\label{def:elliptic}
The operator
$F:\bar{\Sf}\times\R\times\mathcal{S}^2\to\R$
is \emph{degenerate elliptic} if 
\[ F(x,u,X) \leq F(x,v,Y) \]
whenever $u \leq v$ and $X \geq Y$.
\end{definition}

The PDE operators that we consider in this work are degenerate elliptic if they are non-decreasing functions of the argument $u$ and non-increasing functions of all subsequent arguments (which involve second directional derivatives).

Since degenerate elliptic equations need not have classical solutions, solutions may need to be interpreted in a weak sense.  Moreover, even when smooth solutions exist, the use of weak solutions often provides an easier path to convergent numerical methods for fully nonlinear elliptic equations. The numerical methods developed in this article are guided by the very powerful concept of the viscosity solution~\cite{CIL}.

\begin{definition}[Upper and lower semi-continuous envelopes]\label{def:envelope}
The \emph{upper and lower semi-continuous envelopes} of a function $u(x)$ are defined, respectively, by
\[ u^*(x) = \limsup_{\mathbf{y}\to x}u(\mathbf{y}), 
\quad u_*(x) = \liminf_{\mathbf{y}\to x}u(\mathbf{y}). \]
\end{definition}

\begin{definition}[Viscosity solution]\label{def:viscosity}
An upper (lower) semi-continuous function $u$ is a \emph{viscosity subsolution (supersolution)} of~\eqref{eq:PDEelliptic} if for every $\phi\in C^2(\bar{\Omega})$, whenever $u-\phi$ has a local maximum (minimum)  at $x \in \bar{\Omega}$, then
\[ 
F_*^{(*)}(x,u(x),\nabla\phi(x), D^2\phi(x)) \leq (\geq)  0 .
\]
A continuous function $u$ is a \emph{viscosity solution} of~\eqref{eq:PDEelliptic} if it is both a subsolution and a supersolution.
\end{definition}

The first steps towards constructing convergent methods for fully nonlinear elliptic equations where provided by a powerful framework introduced by Barles and Souganidis~\cite{BSNum} and further developed by Oberman~\cite{ObermanSINUM}.

We consider finite difference schemes that have the form
\bq\label{eq:approx} F^h\left(x,u(x),u(x)-u(\cdot)\right) = 0 \eq
where $h$ is a small parameter relating to the grid resolution.

The convergence framework requires notions of consistency, monotonicity and stability.

\begin{definition}[Consistency]\label{def:consistency}
The scheme~\eqref{eq:approx} is \emph{consistent} with the equation~\eqref{eq:PDEelliptic} if for any smooth function $\phi$ and $x\in\bar{\Omega}$,
\[ \limsup_{\epsilon\to0^+,\mathbf{y}\to x,\xi\to0} F^h\left(\mathbf{y},\phi(\mathbf{y})+\xi,\phi(\mathbf{y})-\phi(\cdot) \right) \leq F^*(x,\phi(x),\nabla\phi(x),D^2\phi(x)), 
\]
\[ \liminf_{\epsilon\to0^+,\mathbf{y}\to x,\xi\to0} F^h\left(\mathbf{y},\phi(\mathbf{y})+\xi,\phi(\mathbf{y})-\phi(\cdot)\right) \geq F_*(x,\phi(x),\nabla\phi(x),D^2\phi(x)). \]
\end{definition}

\begin{definition}[Monotonicity]\label{def:monotonicity}
The scheme~\eqref{eq:approx} is monotone if $F^h$ is a non-decreasing function of its final two arguments.
\end{definition}

\begin{definition}[Stability]\label{def:stability}
The scheme~\eqref{eq:approx} is stable if there exists some $M\in\R$ such that if $u^h$ is any solution of~\eqref{eq:approx} then $\|u^h\|_\infty \leq M$.
\end{definition}

To consistent schemes, we can also associate a local truncation error.
\begin{definition}[Truncation error]\label{def:truncation}
The truncation error $\tau(h) > 0$ of the scheme~\eqref{eq:approx} is a quantity chosen so that for every smooth function $\phi$
\[ \limsup\limits_{h\to0}\max\limits_{x\in\G^h}\frac{\abs{F^h(x,\phi(x),\phi(x)-\phi(\cdot)) - F(x,\nabla\phi(x),D^2\phi(x))}}{\tau(h)} < \infty. \]
\end{definition}

Schemes that satisfy these three properties respect the notion of the viscosity solution at the discrete level.  In particular, these schemes preserve the maximum principle and are guaranteed to converge to the solution of the underlying PDE if that equation satisfies a comparison principle.

\begin{definition}[Comparison principle]\label{def:comparison}
A PDE has a \emph{comparison principle} if whenever $u$ is an upper semi-continuous subsolution and $v$ a lower semi-continuous supersolution of the equation, then $u \leq v$ on $\bar{\Omega}$.
\end{definition}

\begin{theorem}[Convergence~\cite{ObermanSINUM}]\label{thm:convergeVisc}
Let $u$ be the unique viscosity solution of the PDE~\eqref{eq:PDEelliptic}, where $F$ is a degenerate elliptic operator with a comparison principle.  Let the approximation $F^h$ be consistent, monotone, and stable and $u^h$ any solution of the scheme~\eqref{eq:approx}.  Then $u^h$ converges uniformly to $u$ as $h\to0$.
\end{theorem}

\subsection{Convergence results for Optimal Transport on the sphere}
The nonlinear PDE~\eqref{eq:OTPDE} that we are interested in does not satisfy a comparison principle, and thus the Barles-Souganidis convergence results do not immediately apply.  However, the authors of this article have recently demonstrated that this framework could be extended to~\eqref{eq:OTPDE} through a careful consideration of the geometry and the introduction of a term that enforces a strong form of stability~\cite{HT_OTonSphere}.

\subsubsection{Reformulation of PDE}\label{sec:PDEFormulation}
The first step in constructing convergent methods is to translate the surface PDE~\eqref{eq:OTPDE} at the point $x_0\in\Sf$ to an equation posed on the local tangent plane $\Tf_{x_0}$.  This requires introducing local coordinates on the tangent plane.  In general, local coordinates will introduce distortions to the Hessian that require the introduction of additional gradient terms.  However, this problem can be avoided with the use of \emph{geodesic normal coordinates}, which preserve distance from the reference point $x_0$.
An explicit expression for these coordinates $v_{x_0}:\Sf \setminus \{-x_0\} \to \Tf_{x_0}$ is given by
\bq\label{eq:normalCoords}
v_{x_0}(x) = x_0\left(1-d_{\Sf}(x_0,x)\cot d_{\Sf}(x_0,x)\right) + x \left(d_{\Sf}(x_0,x)\csc d_{\Sf}(x_0,x)\right).
\eq

For each point $x_0\in \Sf$ we can now define a function $\tilde{u}_{x_0}(z)$ on the relevant tangent plane $\Tf_{x_{0}}$ in a neighborhood of $x_0$ by
\bq\label{eq:tangentFunction}
\tilde{u}_{x_0}(z) = u\left( v_{x_0}^{-1}(z) \right).
\eq
This choice of coordinates allows us to express the PDE~\eqref{eq:OTPDE} at the point $x_0\in\Sf$ as a generalized \MA equation
\bq\label{eq:MATangent}
F(x_0,\nabla\tilde{u}(x_0),D^2\tilde{u}(x_0)) \equiv-\det \left( D^2\tilde{u}(x_0)+A(x_0,\nabla\tilde u(x_0)) \right) + H(x_0,\nabla\tilde u(x_0)) = 0, 
\eq
where all derivatives are now interpreted in the usual sense on a two-dimensional plane.

\begin{remark}
From this point on, we will refer to the surface PDE~\eqref{eq:OTPDE} and the tangent plane representation~\eqref{eq:MATangent} interchangeably, with the implicit understanding that coordinates in the tangent plane are given by the geodesic local coordinates.
\end{remark}

The problem of approximating the Optimal Transport PDE on the sphere can be further simplified by embedding the $c$-convexity (ellipticity) constraint into the equation.  This is accomplished through the introduction of a modified determinant operator satisfying
\begin{equation}\label{eq:detPlus}
\text{det}^{+}(M) = 
\begin{cases}
\text{det}(M), \ \ \ M \geq 0 \\
<0, \ \ \ \text{otherwise}.
\end{cases}
\end{equation}
Then we can absorb the constraint into the PDE~\eqref{eq:MATangent} through the modification
\bq\label{eq:MAPlus}
F^+(x,\nabla u(x),D^2u(x)) \equiv-{\det}^+(D^2{u}(x)+A(x,\nabla u(x))) + H(x,\nabla u(x)) = 0.
\eq

Finally, the solutions of the Optimal Transport PDE satisfy \emph{a priori} Lipschitz bounds $\norm{\nabla u} < R$ for any $R>\pi$ (squared geodesic cost) or $R>C$ (logarithmic cost; see~\cite[Proposition~6.1]{Loeper_OTonSphere} for details).  These bounds can be explicitly built into the PDE through a further modification
\begin{equation}\label{eq:modifiedPDE}
G(x,\nabla u(x),D^2u(x)) \equiv \max \left\{ F^{+}(x,\nabla u(x),D^2u(x)), \norm{ \nabla u(x) } - R \right\} = 0.
\end{equation}
While this modification enforces a condition that is automatically satisfied by solutions at the continuous level, it also improves the stability of approximation schemes.

\subsubsection{Convergence results}
The convergence results of~\cite{HT_OTonSphere} require a mesh or point cloud $\G^h \subset\Sf$ on the sphere that satisfies very mild structural regularity conditions.
We define the discretization parameter $h$ as
\begin{equation}\label{eq:h}
h = \sup\limits_{x\in\Sf}\min\limits_{y\in\G^h} d_{\Sf}(x,y).
\end{equation}
In particular, this guarantees that any ball of radius $h$ on the sphere will contain at least one discretization point.

Then we require a grid that admits a triangulation without any long, skinny triangles.  Specifically:
\begin{hypothesis}[Conditions on point cloud]\label{hyp:grid}
There exists a triangulation $T^h$ of $\G^h$ with the following properties:
\begin{enumerate}
\item[(a)] The diameter of the triangulation, defined as
\bq\label{eq:diam}  \text{diam}(T^h) = \max\limits_{t\in T^h} \text{diam}(t),\eq
 satisfies $\text{diam}(T^h)\to0$ as $h\to0$.
\item[(b)] There exists some $\gamma < \pi$ (independent of $h$) such that whenever $\theta$ is an interior angle of any triangle $t\in T^h$ then $\theta \leq \gamma$.
\end{enumerate}
\end{hypothesis}
We remark that there is no need to construct this triangulation in practice, it need only exist in theory.

We also associate to each point cloud $\G^h$ a search radius $r(h)$ chosen to satisfy
\bq\label{eq:r} r(h)\to0, \, \frac{h}{r(h)} \to 0 \text{ as } h\to 0, \quad \text{diam}(T^h) < r(h). \eq

Next, we project nearby grid points onto the local tangent plan $\Tf_{x_{0}}$, which is spanned by the orthonormal vectors $\left( \hat{\theta},\hat{\phi} \right)$.  For all points $x_i\in\G^h\cap B\left(x_0,r(h) \right)$, we define their projection onto the tangent plane through geodesic normal coordinates via
\bq\label{eq:neighbours}
z_i = x_0 \left(1 - d_{\Sf}(x_0,x_i) \cot d_{\Sf}(x_0,x_i) \right) + x_i \left( d_{\Sf}(x_0,x_i) \csc d_{\Sf}(x_0,x_i) \right).
\eq
Let $\Zf^h(x_0)\subset\Tf_{x_{0}}$ be the resulting collection of points.
See Figure~\ref{fig:sphere}.

\begin{figure}[htp]
\centering
\subfigure[]{
\includegraphics[width=0.5\textwidth,clip=true,trim=0 1in 0 0.8in]{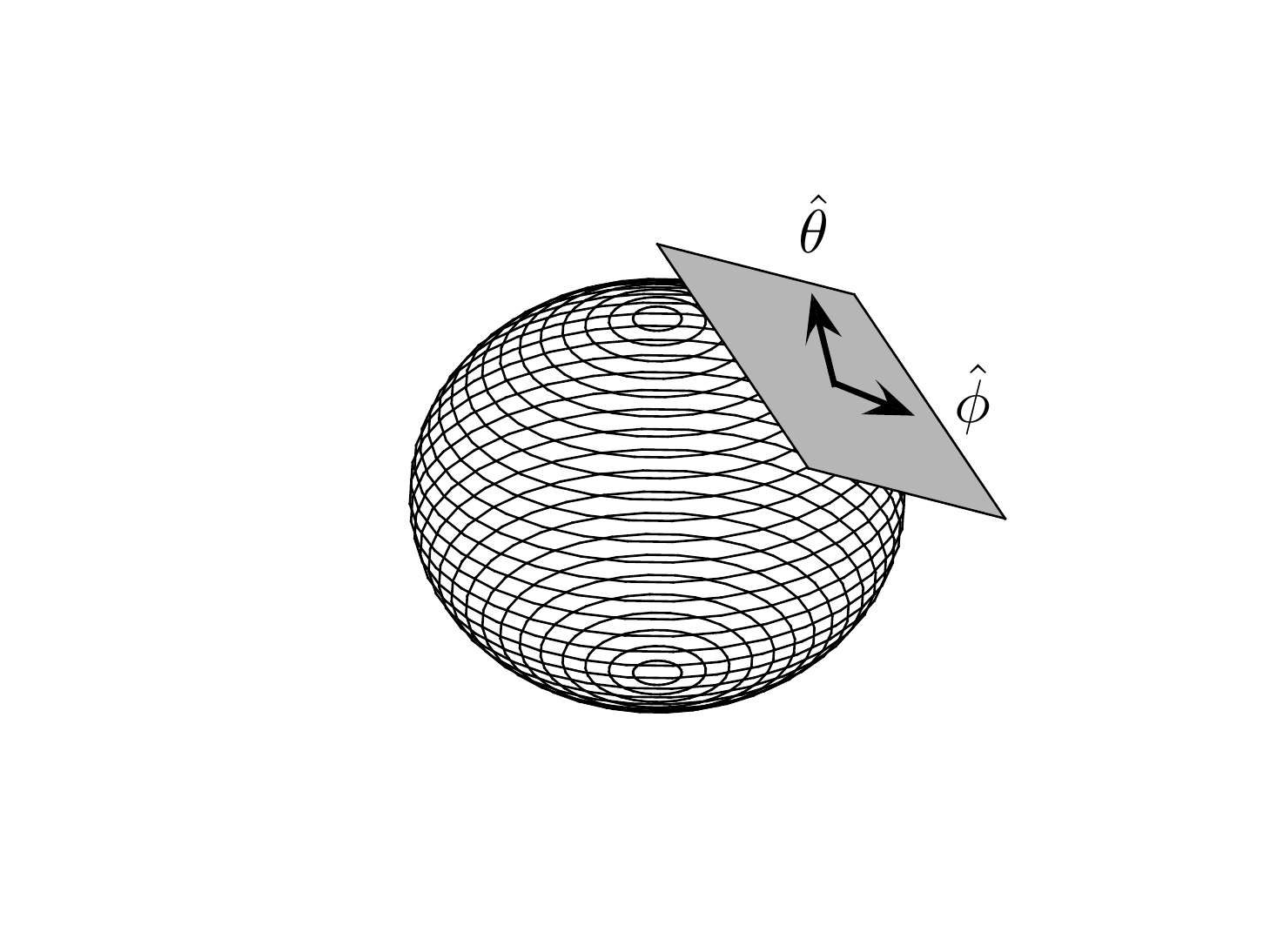}\label{fig:sphere1}} 
\subfigure[]{
\includegraphics[width=0.3\textwidth]{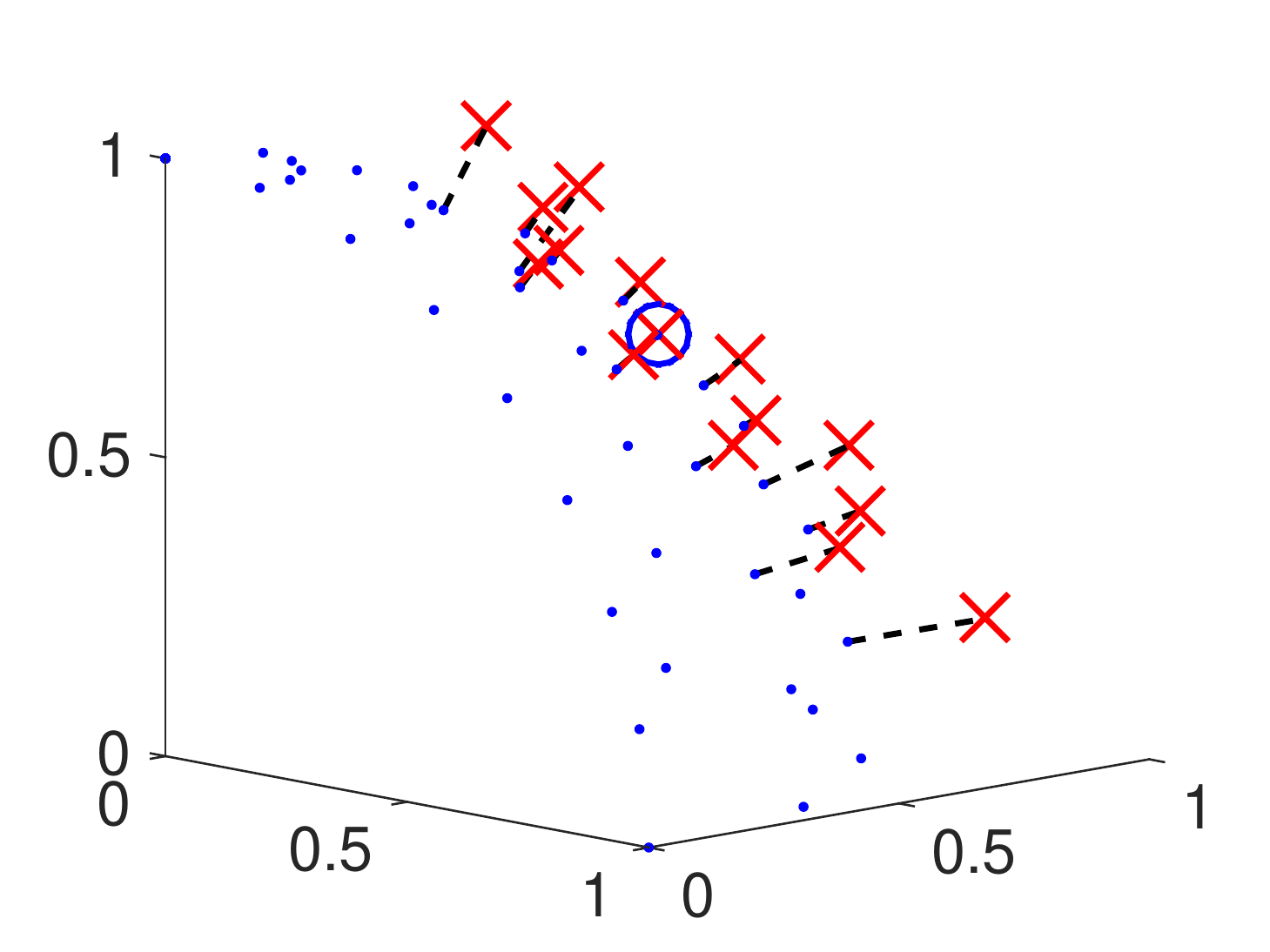}\label{fig:sphere2}}
\caption{\subref{fig:sphere1}~The sphere $\Sf$ and tangent plane $\Tf_{x_{0}}$. \subref{fig:sphere2}~A point cloud discretizing one octant of the unit sphere ($\cdot$), the point $x_0$ (o), and the projections $z$ of neighboring nodes onto $\Tf_{x_{0}}$ ($\times$). }
\label{fig:sphere}
\end{figure}

These are now the discretization points available to use for the approximation of~\eqref{eq:modifiedPDE} at $x_0$; recall that this PDE is posed on the two-dimensional tangent plane.  The convergence result of~\cite{HT_OTonSphere} proposed a specific, simple discretization $E^h$ of the Eikonal term $\|\nabla u(x_0)\|$ that is valid on any grid or point cloud.
\bq\label{eq:eik}
E^h(z,u(z)-u(\cdot)) = \max\limits_{y\in\Zf^h(z)}\frac{ u(z)-u(y) }{\norm{z-y}}.
\eq

Then letting $F^h$ be an approximation of the convexified PDE~\eqref{eq:MAPlus}, we can define an approximation of the modified PDE~\eqref{eq:modifiedPDE} by
\bq\label{eq:approx2}
G^h(z,u(z)-u(\cdot)) = \max\left\{F^h(z,u(z)-u(\cdot)),E^h(z,u(z)-u(\cdot))-R\right\}, \quad z \in \G^h.
\eq

Fixing a point $x^*\in\Sf$ and a sequence $x^*_h\in\G^h \to x^*$, we arrive at the following two-step approach for obtaining a numerical solution $u^h$ of the Optimal Transport problem on the sphere.
\begin{enumerate}
\item[1.] Solve the discrete system
\bq\label{eq:vhScheme}
G^h \left(z,v^h(z)-v^h(\cdot) \right)+\tau(h)v^h(z) = 0, \quad z \in \G^h
\eq
for the grid function $v^h$.
\item[2.] Define the candidate solution
\bq\label{eq:uh}
u^h(z) = v^h(z) - v^h(x^*_h), \quad z \in \G^h.
\eq
\end{enumerate}

Under appropriate assumptions on the density functions $f_1, f_2$ and the approximation scheme $F^h$, this is guaranteed to converge to the solution of the Optimal Transport problem on the sphere.

\begin{theorem}[Convergence of approximation~\cite{HT_OTonSphere}]\label{thm:convergence}
Under the assumptions of Hypothesis~\ref{hyp:Smooth}, let $u\in C^3(\Sf)$ be the unique solution of~\eqref{eq:OTPDE} satisfying $u(x^*) = 0$.  Let $\G^h$ be a grid satisfying Hypothesis~\ref{hyp:grid} and let $F^h$ be any consistent, monotone approximation of~\eqref{eq:modifiedPDE}.  Then for each sufficiently small $h>0$, the grid function $u^h$ defined in~\eqref{eq:uh} is uniquely defined.  Moreover, $u^h$ converges uniformly to $u$ as $h\to0$.
\end{theorem}

\begin{remark}
The above convergence theorem can be extended to weak ($C^1$) solutions by additionally requiring the approximation scheme to underestimate the value of the PDE.
\end{remark}

\section{Formulation of the PDE}\label{sec:PDE}
The PDE~\eqref{eq:OTPDE}-\eqref{eq:cconvex} can be formulated in different equivalent ways.  We begin by briefly describing a reformulation that lends itself to the construction of a discretization that fits within the framework of Theorem~\ref{thm:convergence}.

\subsection{Regularization of logarithmic cost}\label{sec:logCostReg}
One modification of the PDE that we find is necessary to build monotone schemes is to make the logarithmic cost Lipschitz by using a cutoff function. We recall (as discussed in~\autoref{sec:PDEFormulation}) that the solution $u$  to the optimal transport satisfies an \emph{a priori} Lipschitz bound~\cite{Loeper_OTonSphere} $\|\nabla u\| < R$, which also yields the following lower bound on the distance mass can be transported. 
\bq\left\Vert x - y \right\Vert > \frac{1}{\sqrt{R^2+1/4}}\eq\label{eq:logBound}
when $y = T\left( x,\nabla u(x) \right)$ is the exact transport map.

However, the process of solving a discrete version of~\eqref{eq:OTPDE} may evolve through values of $u$ (and consequently $y$) that do not satisfy this bound.  This loss of Lipschitz continuity can lead to a breakdown in monotonicity.  We thus propose the following $C^3$ regularization of the logarithmic cost function, which agrees with the true logarithmic cost when the bound~\eqref{eq:logBound} is satisfied.  
\begin{equation}\label{eq:modlog}
\tilde{c}(x,y) =
\begin{cases}
-\log \left\Vert x - y \right\Vert, \ \ \ \left\Vert x - y \right\Vert \geq \frac{1}{\sqrt{R^2+1/4}} \\
\Psi(\left\Vert x - y \right\Vert), \ \ \ \left\Vert x - y \right\Vert < \frac{1}{\sqrt{R^2+1/4}}
\end{cases}
\end{equation}
where $z_* = \dfrac{1}{\sqrt{R^2+1/4}}$ and 
\[ \Psi(z) = -\log(z_*) - \frac{1}{z_*}(z-z_*)+\frac{1}{2z_*^2}(z-z_*)^2 - \frac{1}{3z_*^3}(z-z_*)^3. \]

Because this regularization does not change the solutions of the PDE, the analysis of~\cite{HT_OTonSphere} and the ultimate convergence result (Theorem~\ref{thm:convergence}) will also apply to discretizations involving this smoother cost function.
\begin{lemma}[Equivalence of solution for modified cost function]\label{lem:equivalence}
Under the assumptions of Hypothesis~\ref{hyp:Smooth}, a function $u\in C^3(\Sf)$ is a solution of~\eqref{eq:OTPDE} with the logarithmic cost~\eqref{eq:logCost} if and only if it is a solution of~\eqref{eq:OTPDE} with the regularized cost~\eqref{eq:modlog}.

\begin{proof}
First let $u$ be a solution using the original cost function $c(x,y) = -\log \left\Vert x - y \right\Vert$.  Because of the \textit{a priori} bounds on the gradient, this automatically satisfies the PDE~\eqref{eq:OTPDE} with the regularized cost function.  

Next let $v$ be any solution of~\eqref{eq:OTPDE} using the regularized cost function $\tilde{c}(x,y)$ from~\eqref{eq:modlog}. Notice that for
 $\left\Vert x - y \right\Vert \geq \frac{1}{\sqrt{R^2+1/4}}$, we have
\begin{equation}
\tilde{c}(x,y) = -\log\left(2 \sin \left( \frac{1}{2} d_{\Sf}(x,y) \right) \right).
\end{equation}
It is easily verified via differentiation that this is convex in $d_{\Sf}(x,y)$ for $\|x-y\| < \dfrac{1}{\sqrt{R^2+1/4}}$, and the original logarithmic cost is also convex in $d_{\Sf}(x,y)$.  The modified cost is $C^3$, so verifying that the second derivative in the variable $d_{\Sf}(x,y)$ is everywhere positive is sufficient to guarantee convexity in this variable. Since this new cost function is Lipschitz and convex in the Riemannian distance, the $\tilde{c}$-convex solution of~\eqref{eq:OTPDE} for the regularized cost $\tilde{c}(x,y)$ is guaranteed to be unique by McCann as noted in~\cite{LoeperReg}.  Since the solution $u$ of the original equation solves this modified equation, uniqueness requires that $v=u$. 
\end{proof}
\end{lemma}

\begin{remark}\label{rmk:nonsmoothlog}
Because this regularization transforms the Optimal Transport problem with a singular cost function into an Optimal Transport problem with a smooth cost function, the techniques of~\cite{HT_OTonSphere} (which were introduced for the squared geodesic cost) can be extended to show that weak ($C^{0,1}$) solutions of this modified problem are also unique.  This assumption of uniqueness of $C^{0,1}$ viscosity solutions was needed to prove convergence even in the smooth setting.  With minor modification, Theorem~\ref{thm:convergence} can also be extended to handle convergence to weak solutions under much milder regularity requirements on the data; see~\cite[Theorem~36]{HT_OTonSphere}. These modifications are addressed in~\autoref{sec:nonsmooth}.
\end{remark}

In the development and analysis of our numerical method in the following sections, we will often refer to the cost function $\tilde{c}$.  In the case of the logarithmic cost, this refers to the regularized cost~\eqref{eq:modlog}.  In the case of the squared geodesic cost, $\tilde{c}$ will refer to the original cost function $c$, which is automatically smooth.

%\begin{remark}
%Many comparison principle and uniqueness results are available for nonlinear PDE. These highly technical arguments adapted to~\ref{eq:modifiedPDE} for the modified cost function $\tilde{c}(x,y)$ are beyond the scope of this article, however. By the techniques in~\cite{HT_OTonSphere}, a smooth function $u$ is a solution of~\ref{eq:modifiedPDE} iff it is a $c$-convex solution of~\ref{eq:OTPDE} for the unmodified cost $c(x,y)$.
%\end{remark}

%\[ \Psi(z) = \log\left(\frac{1}{\sqrt{R^2+1/4}}\right)-\frac{3}{2} + 2z\sqrt{R^2+1/4} - \frac{1}{2}z^2\left(R^2+1/4\right). \]
%where $\Psi^{(n)}(z) = (-1)^{n+1}(R^2+1/4)^{n/2}$, $\Psi(z) \in C^{\infty}[0,1/\sqrt{R^2+1/4}]$. 
%What this achieves is $\tilde{c}(x,y)$ is now Lipschitz and smooth. As shown in~\cite{HT_OTonSphere}, the PDE~\eqref{eq:modifiedPDE} with the logarithmic cost is equivalent to the same PDE with the new cost function $\tilde{c}(x,y)$ for $u \in C^{2}(\mathbb{S}^2)$.

\subsection{Variational formulation of the determinant of a Hessian}\label{sec:var}

Since our PDE involves computing the determinant of a Hessian, here we show how to do this so as to later build a monotone discretization of the second derivatives. As utilized in~\cite{FO_MATheory}, Hadamard's inequality allows the determinant of a positive definite matrix $M$ to be computed via the minimization problem
\begin{equation}
\det M = \min_{v_i \in V} \prod_{i=1}^{d} v_{i}^{T} M v_i 
\end{equation}
where $V$ is the set of all orthogonal bases for $\R^d$.

We require a formulation satisfying~\eqref{eq:detPlus}, which modifies this formulation to ensure that no negative terms appear when the symmetric matrix $M$ is not positive definite.  A simple approach is to use
\begin{equation}\label{eq:detPlusUse}
{\det}^+ M = \min_{v_i \in V} \prod_{i=1}^{d} \max\left\{v_{i}^{T} M v_i, 0\right\}. 
\end{equation}

If our matrix $M = D^2\phi(x)$ is a Hessian matrix, this becomes:
\begin{align*}
{\det}^+(D^2 \phi (x)) &= \min_{\nu_1, \nu_2 \in V} \left\{\prod_{j=1}^2 \nu_j^T (D^2 \phi(x)) \nu_j, 0 \right\}\\
  &= \min_{\nu_1, \nu_2 \in V} \prod_{j=1}^2 \max \left\{ \frac{\partial^2 \phi}{\partial \nu_j^2}, 0 \right\}.
\end{align*}

In particular, this allows us to replace the determinant in~\eqref{eq:OTPDE} with
\bq\label{eq:detPDE}
{\det}^+(D^2u(x) + A(x,\nabla u(x))) = \min_{\nu_1, \nu_2 \in V} \prod_{j=1}^2 \max \left\{ \frac{\partial^2 u(x)}{\partial \nu_j^2} + \left.\frac{\partial^2c(x,y)}{\partial\nu_j^2}\right|_{y = T(x,\nabla u(x))}, 0 \right\}.
\eq

\subsection{Mixed Hessian}\label{sec:mixedHessian}

The framework developed by the authors of this article in~\cite{HT_OTonSphere} only requires the construction of consistent approximations of derivatives with respect to $x$, expressed in the geodesic normal coordinate system. For the mixed Hessian term, $\det D^2_{xy} c(x,y)$ it is not immediately clear how to do this. In fact, the relative simplicity of the notation obfuscates the actual complexity of the object. In~\cite{LoeperReg}, Loeper outlines a clearer representation of this quantity in curved geometries and for different cost functions. 

We recall that the optimal map $T(x,p)$ (also known as the $c$-exponential map) satisfies~\eqref{eq:mapConditionSphere} and can be constructed explicitly for the cost function of interest to us via~\eqref{eq:mapd2}-\eqref{eq:mapLog}.  Then from~\cite{LoeperReg}, the mixed Hessian satisfies
\begin{equation}\label{representationformula}
[D^2_{xy}c]^{-1} = -D_{p} T(x,p) \vert_{x,p = -\nabla_{x}c(x,y)}.
\end{equation}

%In the sequel, we show that due to the construction of the Lax-Friedrichs-type Laplacian regularization term~\ref{laxfriedrichs}, in reality, we do not need to bend over backwards to make a monotone discretization for this term. What we need, however, is any way to consistently approximate this term. We decide to derive its explicit formulation and use this as a consistent approximation.

This representation formula~\eqref{representationformula} shows that the inverse of the mixed Hessian is simply the Jacobian of the map $T(x,p)$ with respect to $p$. Since we will be taking the determinant, this is actually a change of area formula  for the transformation $T(x,p)$. See Figure~\ref{fig:mixedhessian}.

\begin{figure}[htp]
\includegraphics[height = 6.5cm]{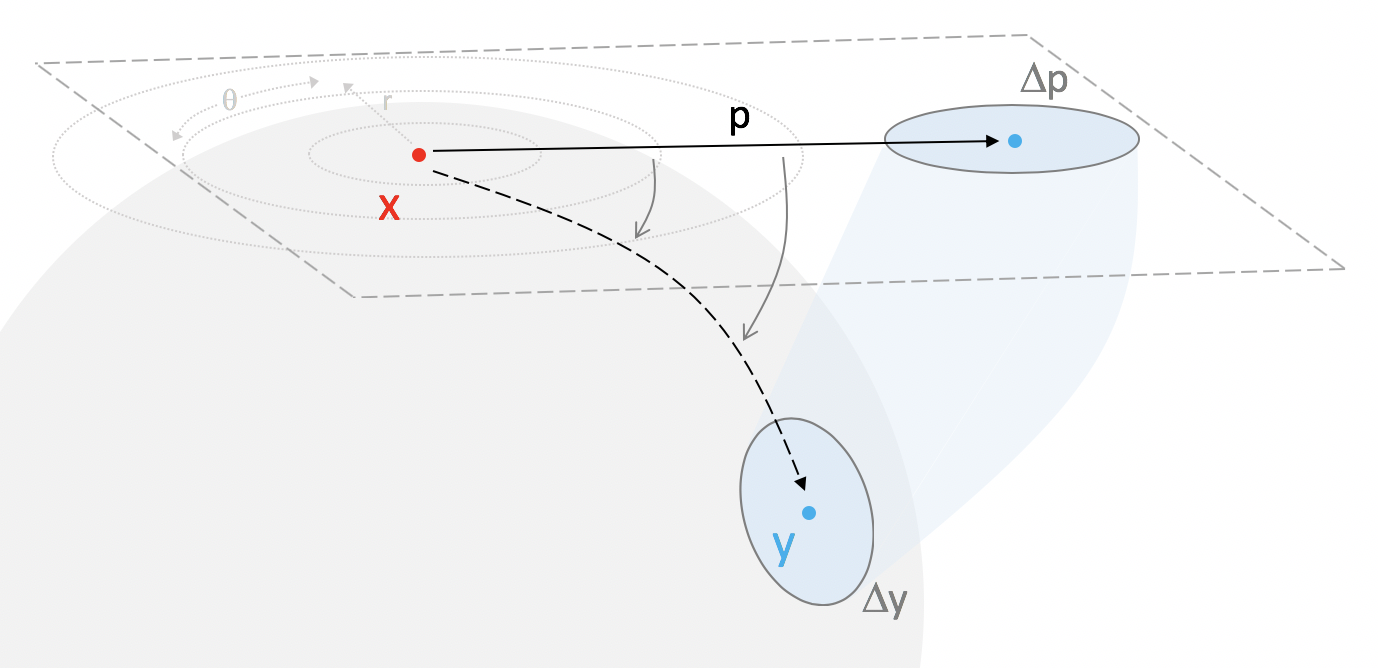}
\caption{The determinant of the mixed Hessian is the change in area formula from the set $T(x,E)$ on the sphere sphere to the set $E$ on the local tangent plane}
\label{fig:mixedhessian}
\end{figure}

We will compute this change of area by computing the linear differential map $dT$ for both cost functions. That is, the change in area will be computed using orthogonal perturbations $\Delta p_1$, $\Delta p_2$ (so $\Delta p_1 \cdot \Delta p_2 = 0$) in the tangent plane $\mathcal{T}_x$.
\begin{align*}
 \left|\det\right.&\left.(D_pT(x,p))\right| = \\ & \left\vert \lim_{\left\Vert \Delta p_1 \right\Vert \rightarrow 0} \frac{T(x,p+\Delta p_1)-T(x,p)}{\| \Delta p_1\|} \times \lim_{\left\Vert \Delta p_2 \right\Vert \rightarrow 0} \frac{T(x,p+\Delta p_2)-T(x,p)}{\|\Delta p_2\|} \right\vert. \end{align*}

In general, the area element on a manifold is a function of the wedge product of two covectors, which need not be orthogonal.  This has the interpretation of the area of a parallelogram on the manifold.  However, in the special case where the vectors are orthogonal (or are orthogonal to leading order), this reduces to an ordinary product.  This is indeed the case for both the squared geodesic and logarithmic cost functions.  Thus the change of area formula reduces to the simpler expression
\begin{align*}
 \left|\det\right.&\left.(D_pT(x,p))\right| = \\ &\lim_{\left\Vert \Delta p_1 \right\Vert, \left\Vert \Delta p_2 \right\Vert \rightarrow 0, \Delta p_1 \cdot \Delta p_2 = 0} \frac{d_{\mathbb{S}^2} \left( T(x,p), T(x,p+ \Delta p_1) \right) d_{\mathbb{S}^2} \left( T(x,p), T(x,p + \Delta p_2) \right) }{\left\Vert \Delta p_1 \right\Vert \left\Vert \Delta p_2 \right\Vert}
\end{align*}
and the determinant of the mixed Hessian is given by
\bq\label{eq:mixedHessian}
\abs{\det(D^2_{xy}c(x,y))} = \left.\frac{1}{\abs{D_{p} T(x,p)}} \right|_{p = -\nabla_{x}c(x,y)}.
\eq

%\begin{equation}
%\lim_{\left\Vert \Delta p \right\Vert \rightarrow 0} \frac{d_{\mathbb{S}^2} \left( T(x,p), T(x,p+ \Delta p) \right)}{\left\Vert \Delta p \right\Vert}
%\end{equation}

%In general, the area element on a manifold is a function of the wedge product of two not necessarily orthogonal covectors: $dx \wedge dy: \mathcal{T}_x \times \mathcal{T}_x \rightarrow \mathbb{R}$. Supposing the covectors are orthogonal, then we simply compute $dxdy$. Covectors and covector wedge products transform via the inverse of the pullback of the diffeomorphic differential map $dT^{*}$ from $\mathcal{T}_y^{*}$ to $\mathcal{T}_p^{*}$, which coincides with the pushforward map $dT$. Thus, if orthogonal covectors transform to orthogonal covectors, to compute the area element via the transformation of coordinates $T$ from the tangent plane to the sphere we simply compute the product $dT(dx) dT(dy)$ for orthogonal covectors.
%
%We compute $dT$ by taking the derivative of tangent vectors with respect to orthogonal perturbations and then taking their product. Since the mixed Hessian is the inverse of this change of area, which then take its reciprocal. Thus, the mixed Hessian can be computed via:

This can be computed explicitly for both the squared geodesic cost,
\begin{equation}\label{eq:mixedd2}
\abs{\det (D^2_{xy} c(x,y)) } = \frac{\left\Vert p \right\Vert}{\sin \left\Vert p \right\Vert},
\end{equation}
and for the logarithmic cost,
\begin{equation}\label{eq:mixedLog}
\abs{\det (D^2_{xy} c(x,y)) } = \left( \left\Vert p \right\Vert^2 + 1/4 \right)^2.
\end{equation}
See Appendix~\ref{app:mixedHessian} for details.

We note also that these formulas coincide with the formulas that can be derived via standard change of variables formulas by requiring
\[ \int_{T(x,E)}dS = \int_E \abs{\det(D_pT(x,p))}\,dp \]
for every measurable $E\in\mathcal{T}(x)$.

\section{Numerical Method}\label{sec:scheme}
We now explain how we actually construct an approximation scheme for~\eqref{eq:OTPDE} at a point $x\in\Sf$.  In~\autoref{sec:convergence}, we will demonstrate that this method does converge to the true solution of the optimal transport problem.

\subsection{Construction of finite difference stencils}\label{sec:stencils}
We begin with a point cloud $\G \in \Sf$ that discretizes the sphere; we assume only the minimal regularity required by Hypothesis~\ref{hyp:grid}.

We begin by considering a fixed point $x_i \in \G$ and establishing a computational neighborhood $N(i)$ about this point. For any fixed $C>0$, we define
\begin{equation}\label{eq:neighborhood}
N(i) = \{j \mid x_j\in\G, \, d_{\mathbb{S}^2}(x_i, x_j) \leq C\sqrt{h}  \}.
\end{equation}

Once $N(i)$ is established, the points $x_j \in N(i)$ are projected on to the local tangent plane $\mathcal{T}_{x_{i}}$ via a geodesic normal coordinate projection~\eqref{eq:normalCoords}.  
\[z_j = v_{x_i}(x_j). \] 
We denote the resulting point cloud on the tangent plane by
\[ \Nf(i) = \{z_j \mid j \in N(i)\} \subset \Tf_{x_{i}}. \]
Figure~\ref{fig:tangentplane} shows an exaggerated example of this tangent plane projection.

\begin{figure}[htp]
	\includegraphics[height=11cm]{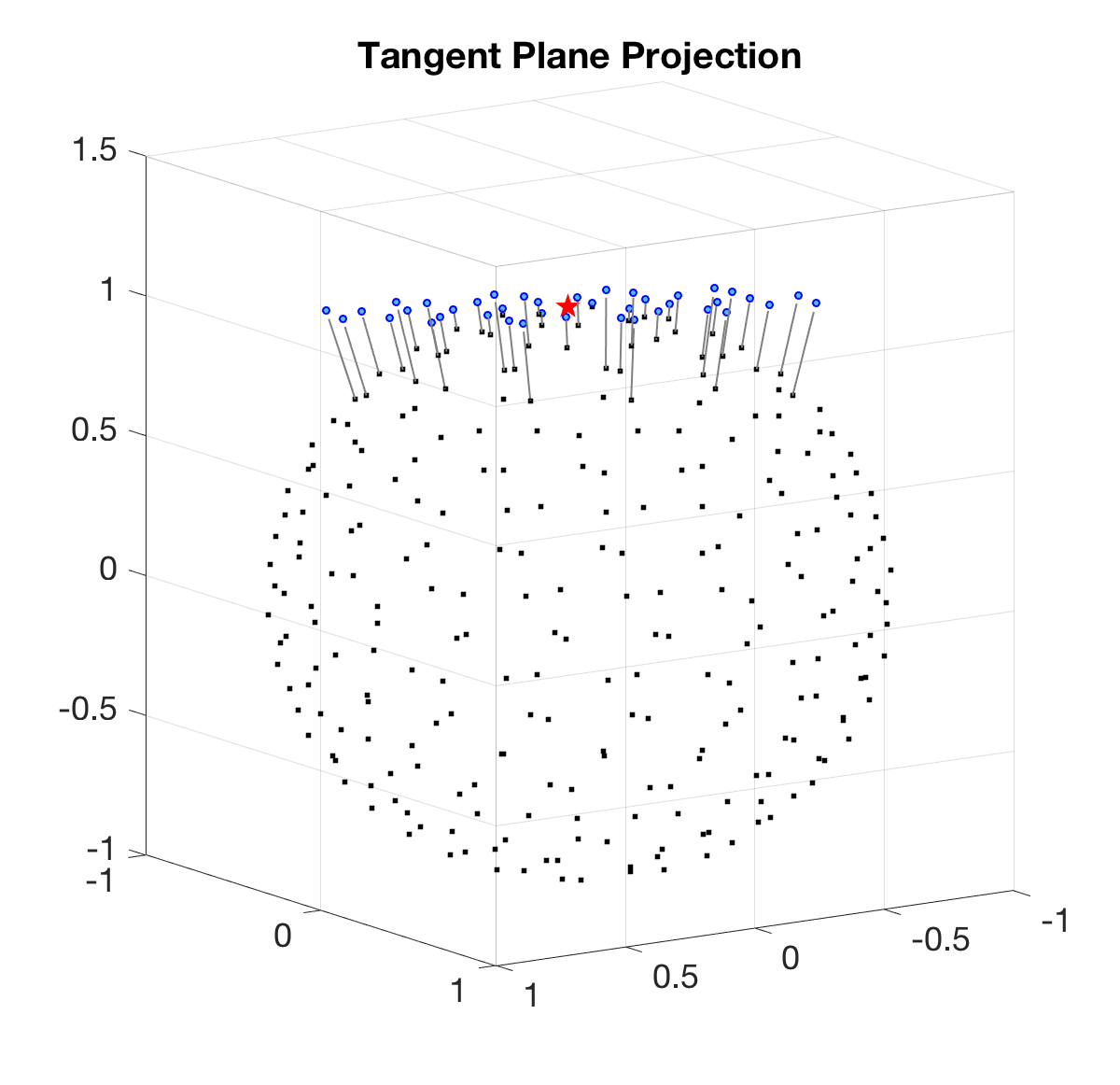}
	\caption{Projection onto the tangent plane via geodesic normal coordinates. The points in $\G$ are indicated with small black squares, the computational point $x_i$  is indicated as a red star, and the projection of the neighborhood $N(i)$ to the local tangent plane is indicated with blue circles.}
\label{fig:tangentplane}
\end{figure}

Next we suppose that we are interested in resolving behavior along some direction $\nu\in\R^2$.  Following a slight modification of~\cite{FroeseMeshfreeEigs}, we select four points $x_{\nu,j} \in \Nf(i)$ that are well-aligned with the direction $\nu$.  See Figure~\ref{fig:neighbors}.

We introduce the following notation:
\begin{itemize}
\item $r = C\sqrt{h}$ is the search radius used to define the neighborhood $N(i)$.  For all $x_{\nu,j}\in \Nf(i)$ we have that $\| x_{\nu,j}-x_i\| \leq r$.
\item $\theta_{\nu,j}$ is the angle between $x_{\nu,j}-x_i$ and the direction $\nu$.
\item $d \theta_{\nu, j}$ is the minimal absolute angle between $x_{\nu,j}-x_i$ and the direction ${\nu}$
\item $d\theta$ is the overall angular resolution of the stencil, which is related to the search radius through $d\theta = \dfrac{2h}{r} + \bO(h) = \bO(\sqrt{h})$; see~\cite{FroeseMeshfreeEigs}.
\item Each $x_{\nu,j}$ can be represented in polar coordinates as $(h_{\nu,j},\theta_{\nu,j})$ using the coordinate system where $x_i$ is the origin and the coordinate directions $\nu,\nu^\perp$ are orthogonal.
\item Components of $x_{\nu,j}$ can be expressed using the shorthand notation
\[ C_{\nu,j} = h_{\nu,j}\cos \theta_{\nu,j}, \quad S_{\nu,j} = h_{\nu,j}\sin \theta_{\nu,j}. \]
\end{itemize}

The following lemma follows immediately from the proof of~\cite[Lemma~11]{FroeseMeshfreeEigs}.  Figure~\ref{fig:unifElliptic} illustrates the four small balls where each of these four neighbors is required to exist.
\begin{lemma}[Properties of neighbors]\label{lem:neighbors}
For every $x_i\in\G$ and $\nu\in\R^2$, four neighbors $x_{\nu,j}\in\Nf(i)$ exist satisfying the following properties:
\begin{itemize}
\item $x_{\nu,j}$ resides in the $jth$ quadrant.
\item The angular component of $x_{\nu,j}$ satisfies $d\theta \leq d\theta_{\nu,j} \leq 2d\theta$.
\item The radial component of $x_{\nu,j}$ satisfies $r-2h \leq h_{\nu,j} \leq r$.
\end{itemize} 
\end{lemma}
These additional requirements on stencil will be critical to developing monotone approximations of functions of the gradient.

\begin{figure}[htp]
	\subfigure[]{\includegraphics[width=0.45\textwidth]{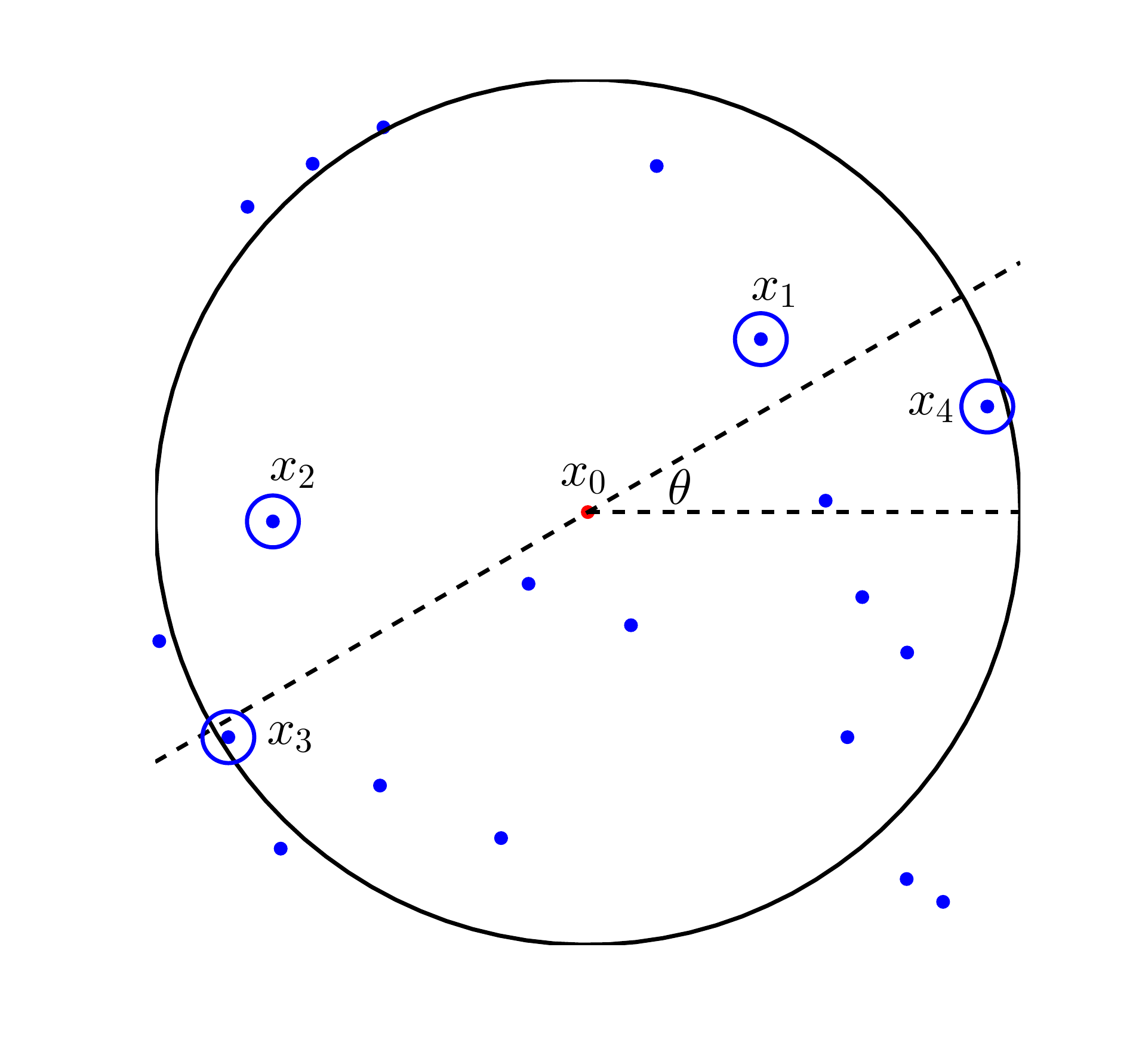}\label{fig:neighbors}}
	%\subfigure[]{\includegraphics[height=4.5cm]{MeshfreeStencil}}
	\subfigure[]{\includegraphics[width=0.5\textwidth]{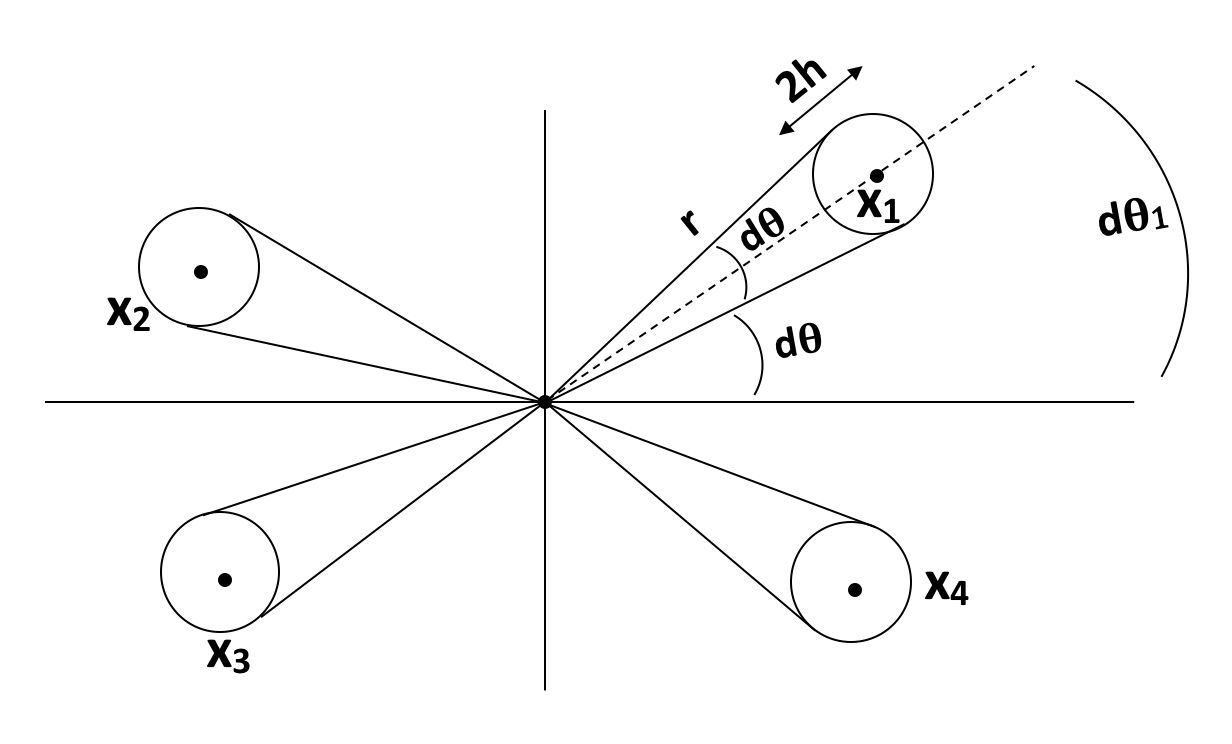}\label{fig:unifElliptic}}
	\caption{Choice of computational points $x_{\nu,j}$ in the local tangent plane.}
	\label{fig:meshfree}
\end{figure}

\subsection{Approximation of second derivatives}
Our overall approximation of~\eqref{eq:OTPDE} will hinge on the construction of (negative) monotone schemes for second directional derivatives $\dfrac{\partial^2\phi}{\partial\nu^2}$.  We introduce approximations of the form
\begin{equation}
\Dt_{\nu\nu}\phi(x_i) =  \sum_{j=1}^{4} a_{\nu,j}\left( \phi(x_{\nu,j}) - \phi(x_i) \right).
\end{equation}

As in~\cite{FroeseMeshfreeEigs}, consistency and negative monotonicity can be achieved by finding a solution of the system
\begin{equation}
\begin{cases}
\sum a_{\nu,j} h_{\nu,j} \cos \theta_{\nu,j} = 0 \\
\sum a_{\nu,j} h_{\nu,j} \sin \theta_{\nu,j} = 0 \\
\sum \frac{1}{2}a_{\nu,j} h_{\nu,j}^2 \cos^2 \theta_{\nu,j} = 1 \\
a_{\nu,j} \geq 0
\end{cases}
\end{equation}

An explicit solution is given by
\bq\label{eq:coeffs}
\begin{split}
a_{\nu,1} &= \frac{2S_{\nu,4}(C_{\nu,3}S_{\nu,2}-C_{\nu,2}S_{\nu,3})}{\det(A)}\\
a_{\nu,2} &= \frac{2S_{\nu,3}(C_{\nu,1}S_{\nu,4}-C_{\nu,4}S_{\nu,1})}{\det(A)}\\
a_{\nu,3} &= \frac{-2S_{\nu,2}(C_{\nu,1}S_{\nu,4}-C_{\nu,4}S_{\nu,1})}{\det(A)}\\
a_{\nu,4} &= \frac{-2S_{\nu,1}(C_{\nu,3}S_{\nu,2}-C_{\nu,2}S_{\nu,3})}{\det(A)}
\end{split}
\eq
where \bq\label{eq:detA}\begin{split}\det(A) = &(C_{\nu,3}S_{\nu,2}-C_{\nu,2}S_{\nu,3})(C_{\nu,1}^2S_{\nu,4}-C_{\nu,4}^2S_{\nu,1})\\&-(C_{\nu,1}S_{\nu,4}-C_{\nu,4}S_{\nu,1})(C_{\nu,3}^2S_{\nu,2}-C_{\nu,2}^2S_{\nu,3}).\end{split}\eq

\subsection{Approximation of functions of the gradient}
The PDE~\eqref{eq:OTPDE} involves several terms of the form $g(\nabla u)$.  Importantly, either automatically or through appropriate regularization (see \autoref{sec:logCostReg}), each of these functions $g$ has a bounded Lipschitz constant $L_g$.  This allows us to pursue a generalized Lax-Friedrichs type discretization of the form
\bq\label{eq:discGrad}
\begin{split}
\tilde{g}^\pm(\Dt \phi(x_i)) = & g\left(\sum\limits_{\nu\in\{(1,0),(0,1)\}}\nu\sum\limits_{j=1}^4 b_{\nu,j}\left(\phi(x_{\nu,j})-\phi(x_i)\right)\right)\\ &\mp \epsilon_g\sum\limits_{\nu\in\{(1,0),(0,1)\}}\sum\limits_{j=1}^4a_{\nu,j}\left(\phi(x_{\nu,j})-\phi(x_i)\right).
\end{split}
\eq
Above, the coefficients $a_{\nu,j}$ are identical to the coefficients that arise in the approximation of second directional derivatives.  This introduces a Laplacian regularization term, which is carefully chosen to enforce monotonicity (or negative monotonicity) even if the coefficients $b_{\nu,j}$ do not on their own produce a monotone scheme.

We first require coefficients $b_{\nu,j}$ that ensure that
\[ \Dt_{\nu}\phi(x_i) = \sum\limits_{j=1}^4 b_{\nu,j}\left(\phi(x_{\nu,j})-\phi(x_i)\right)\]
is a consistent approximation of the first directional derivative $\dfrac{\partial\phi(x_i)}{\partial\nu}$.

Taylor expanding, we obtain
\[
\Dt_{\nu}\phi(x_i)= \sum_{j=1}^{4} b_{\nu,j} \left(h_{\nu,j} \cos \theta_{\nu,j} \frac{\partial\phi(x_i)}{\partial\nu} + h_{\nu,j} \sin \theta_{\nu,j} \frac{\partial\phi(x_i)}{\partial\nu^\perp} + \mathcal{O}(r^2)\right).
\]
Consistency then requires a solution of the system
\begin{equation}\label{eq:coeffsGrad}
\begin{cases}
\sum b_{\nu,j} h_{\nu,j} \cos \theta_{\nu,j} = 1 \\
\sum b_{\nu,j} h_{\nu,j} \sin \theta_{\nu,j} = 0.
\end{cases}
\end{equation}

An explicit solution is
\begin{equation}
\begin{split}
b_{\nu,1} = \frac{S_{\nu,4} (S_{\nu,3} C_{\nu,2}^2 - S_{\nu,2} C_{\nu,3}^2)}{\det(A)} \\
b_{\nu,2} = -\frac{S_{\nu,3} (S_{\nu,4} C_{\nu,1}^2 - S_{\nu,1} C_{\nu,4}^2)}{\det(A)} \\
b_{\nu,3} = \frac{S_{\nu,2} (S_{\nu,4} C_{\nu,1}^2 - S_{\nu,1} C_{\nu,4}^2)}{\det(A)} \\
b_{\nu,4} = -\frac{S_{\nu,1} (S_{\nu,3} C_{\nu,2}^2 - S_{\nu,2} C_{\nu,3}^2)}{\det(A)}
\end{split}
\end{equation}
where $\det(A)$ is again given by~\eqref{eq:detA}.

We then substitute these coefficients into~\eqref{eq:discGrad} and define a regularization factor satisfying
\bq\label{eq:epsilon}
\epsilon_g =  \max\left\{\frac{L_g\abs{b_{\nu,j}}}{a_{\nu,j}} \mid j\in\{1,2,3,4\}, \, \nu \in \{(0,1),(1,0)\}\right\}.
\eq
We will verify that this is finite and bounded in~\autoref{sec:convergence}.

\subsection{Approximation of the nonlinear operator}
We now have the building blocks in place to describe a discretization of the full nonlinear operator~\eqref{eq:OTPDE} at the point $x_i\in\G$.

The variational formulation of the modified determinant~\eqref{eq:detPlusUse} requires performing a minimization over the set $V$ of orthogonal bases for $\R^2$.  In the discrete version, we consider a finite subset of $V$ that ensures that all directions are resolved in the limit $h\to0$.  A simple choice is given by
\bq\label{eq:directions}
\tilde{V} = \left\{(\cos\theta,\sin\theta) \mid \theta = jd\theta, \, j = 0, \ldots, \frac{\pi}{2d\theta}\right\}
\eq
where $d\theta = \bO\left(\sqrt{h}\right)$ is the angular resolution of the stencil described in~\autoref{sec:stencils}.

The PDE involves several different functions of the gradient. For compactness, we introduce the shorthand notation
\bq\label{eq:gradFunctions1}
\begin{split}
g_{1,\nu}(x_i,p) = \Dt_{\nu\nu}\tilde{c}(x_i,T(x_i,p))
\end{split}
\eq
where the differencing is performed only in the first argument of $\tilde{c}$.  This involves the explicit formulas for the optimal map given in~\eqref{eq:mapd2}-\eqref{eq:mapLog}.  

 We also define
\bq\label{eq:gradFunctions2}
\begin{split}
g_2(x_i,p) &= \frac{\abs{\det D^2_{xy}c\left(x_i,T(x_i,p)\right)}}{f_2\left(T(x_i,p)\right)}
\end{split}
\eq
recalling that the determinant of the mixed Hessian can be replaced with the simple explicit representations obtained in~\autoref{sec:mixedHessian}, which we here denote by $H(p)$.

The discretization of these functions of the gradient require a regularization parameter~\eqref{eq:epsilon}, which involves Lipschitz bounds on these functions.  We select bounds satisfying
\bq\label{eq:L1}
L_{g_{1}} > L_{\tilde{c}} L_T
\eq 
and
\bq\label{eq:L2}
L_{g_{2}} \geq \left\|\frac{1}{f_2}\right\|^2\left(\|f_2\| L_{H} + \| H\| L_{f_2}L_T\right)
\eq
where  $L_T$ is the Lipschitz constant of the optimal map $T(x,p)$ with respect to the variable $p$, $L_H$ and $L_{f_2}$ are the Lipschitz constants of the functions $H$ and $f_2$ respectively, and
\bq\label{eq:Lc} L_{\tilde{c}} = \max\limits_{\abs{\nu}=1,x,y\in\Sf} \left\| \nabla_y \left(\nu^T D_{xx}^2\tilde{c}(x,y) \nu\right) \right\|. \eq
We remark that while these constants $L_{g_{1}}$ and $L_{g_2}$ depend on the problem data and particular cost function, they are all guaranteed to be bounded under the assumptions of Hypothesis~\ref{hyp:Smooth} and can be computed explicitly using the formulas in~\eqref{eq:mapd2},\eqref{eq:mapLog}, \eqref{eq:modlog}, \eqref{eq:mixedd2}, and \eqref{eq:mixedLog}.

We can then write down the full discretization of~\eqref{eq:OTPDE} as
\bq\label{eq:discOTPDE}
\begin{split}
F^h(x,u(x)-u(\cdot)) = &-\min\limits_{(\nu_1,\nu_2)\in\tilde{V}}\prod\limits_{j=1}^2 \max\left\{ \Dt_{\nu_j\nu_j}u(x_i)+\tilde{g}_{1,\nu_j}^-(\Dt u(x_i)), 0\right\}\\ &+ f_1(x_i) \tilde{g}_2^+\left(x_i,\Dt u(x_i)\right).\end{split}
\eq

\subsection{Solution method}
In order to efficiently obtain a convergent approximation to~\eqref{eq:OTPDE}, we will slightly modify the two-step procedure described
in~\eqref{eq:vhScheme}-\eqref{eq:uh}.  We propose instead the following solution and verification process, which is equivalent.
\begin{enumerate}
\item[1.] Solve the discrete system
\bq\label{eq:vhScheme}
F^h(x,v^h(x)-v^h(\cdot))+\sqrt{h}v^h(x) = 0, \quad x \in \G
\eq
for the grid function $v^h$.
\item[2.] Verify that the grid function $v^h$ satisfies the bounds
\[ E^h(x,v^h(x)-v^h(\cdot)) \leq R, \quad x \in \G. \]
\item[3.] If the verification step fails, redefine $v^h$ by solving the modified system~\eqref{eq:vhScheme}
\[G^h \left(x,v^h(x)-v^h(\cdot) \right)+\sqrt{h}v^h(x) = 0, \quad x \in \G
\]
using the solution obtained in Step 1 as an initial guess.
\item[4.] Define the discrete solution
\bq\label{eq:uh2}
u^h(x) = v^h(x) - v^h(x^*_h), \quad x \in \G.
\eq
\end{enumerate}

We note that in practice, we have never found Step 3 above to be necessary.  Thus although this procedure appears to be longer than simply performing Steps~3-4, it actually allows us to obtain the same solution by solving a simpler system.

The strong nonlinearity in the PDE~\eqref{eq:OTPDE}, particularly in that it involves nonlinear gradient terms that have very little required structure, makes the construction of a nonlinear Gauss-Jacobi, algebraic multigrid, and/or approximate Newton-type method highly nontrivial. In the present work, we perform all our computations using explicit parabolic schemes of the form
\[v_{n+1}^h(x_i) = v_n^h(x_i) - \Delta t F^h(v_n^h(x_i), v_n^h(x_i)-v_n^h(\cdot)).\]

As discussed in~\cite{ObermanSINUM}, $\Delta t$ has to satisfy a nonlinear CFL condition in order to guarantee convergence. In particular, we require $\Delta t < 1/L_{F^{h}}$, where $L_{F^{h}}$ is the Lipschitz constant of $F^h$ with respect to the arguments $u^h_i$. This Lipschitz constant scales like $L_{F^{h}} = \mathcal{O}(h^{-2})$ and can either be determined explicitly \emph{a priori} or adaptively under a requirement that the residual should decrease.  In some cases, acceleration of this process is possible using the approach of~\cite{SchaefferHou}.

Faster solvers for these kinds of systems are, of course, desirable and will be explored in future work.

\section{Convergence}\label{sec:convergence}
We are now prepared to prove that the numerical method defined by~\eqref{eq:discOTPDE} and the subsequent solution procedure converges.
\begin{theorem}[Convergence]\label{thm:convergenceScheme}
Under the assumptions of Hypothesis~\ref{hyp:Smooth}, let $u\in C^3(\Sf)$ be the unique solution of~\eqref{eq:OTPDE} satisfying $u(x^*) = 0$.  Let $\G^h$ be a grid satisfying Hypothesis~\ref{hyp:grid} and let $F^h$ be defined as in~\eqref{eq:discOTPDE}.  Then for each sufficiently small $h>0$, the grid function $u^h$ defined in~\eqref{eq:uh2} is uniquely defined.  Moreover, $u^h$ converges uniformly to $u$ as $h\to0$.
\end{theorem}

This result follows immediately from the framework described in Theorem~\ref{thm:convergence} provided we can verify that our approximation scheme $F^h$ is consistent (Lemma~\ref{lem:consistent}) and monotone (Lemma~\ref{lem:monotone}).  This will be accomplished in several lemmas throughout the remainder of this section.

\subsection{Bounds on coefficients}
We begin by demonstrating that the coefficients $a_{\nu,j}, b_{\nu,j}$ appearing in the approximation of the second directional derivatives can be bounded.

\begin{lemma}[Bounds on coefficients (second derivatives)]\label{lem:bounds2}
There exists a constant $C>0$ such that for all sufficiently small $h>0$ and $\nu\in\R^2$, the coefficients defined by~\eqref{eq:coeffs} satisfy
\[  a_{\nu,j} \geq \frac{C}{h} .\]
\end{lemma}
\begin{proof}
We establish a bound for the coefficient $a_{\nu,1}$; the remaining coefficients are similar.

Recall the notation $C_{\nu,j} = h_{\nu,j}\cos \theta_{\nu,j}$, $S_{\nu,j} = h_{\nu,j}\sin \theta_{\nu,j}$.  Since each $x_{\nu,j}$ lies in the $jth$ quadrant, each of these terms has a definite sign.  Based on the requirements on $h_{\nu,j}, d\theta_{\nu,j}$ established in Lemma~\ref{lem:neighbors}, we can record the asymptotic bounds
\begin{align*}
 (r-2h)\left(1-\bO(d\theta^2)\right) &\leq \abs{C_{\nu,j}} \leq r\\
   (r-2h)\left(d\theta - \bO(d\theta^3)\right) &\leq \abs{S_{\nu,j}} \leq r\left(2d\theta + \bO(d\theta^3)\right).
%\cos d\theta_{\nu,j} = \pm 1 + \mathcal{O}(d\theta^2) = \pm 1 + \mathcal{O}(h)\\
%d \theta + \mathcal{O}(d \theta^3) \leq |\sin d\theta_{\nu,j} | \leq 2 d \theta + \mathcal{O}(d \theta^3)
\end{align*}
We recall also that $r,d\theta = \bO(\sqrt{h})$.

These observations allow us to establish the following bounds:
\[
\begin{split}\det(A) = &(C_{\nu,3}S_{\nu,2}-C_{\nu,2}S_{\nu,3})(C_{\nu,1}^2S_{\nu,4}-C_{\nu,4}^2S_{\nu,1})\\&-(C_{\nu,1}S_{\nu,4}-C_{\nu,4}S_{\nu,1})(C_{\nu,3}^2S_{\nu,2}-C_{\nu,2}^2S_{\nu,3})\\
&\leq r^5(4d\theta + \bO(d\theta^3))(4d\theta+\bO(d\theta^3)) + r^5(4d\theta + \bO(d\theta^3))(4d\theta+\bO(d\theta^3))\\
%&\leq (2r^2d\theta+2r^2d\theta)(2r^3d\theta+2r^3d\theta) + (2r^2d\theta+2r^2d\theta)(2r^3d\theta+2r^3d\theta)\\
&= 32r^5d\theta^2 + \bO(h^{9/2})
\end{split}
\]
and 
\[
\begin{split}
-2S_{\nu,4}(C_{\nu,2}S_{\nu,3} - C_{\nu,3}S_{\nu,2}) &\geq 2(r-2h)^3(d\theta-\bO(d\theta^3))\left(2d\theta-\bO(d\theta^3)\right)\\
  &= 4r^3d\theta^2 + \bO(h^3).
\end{split}
\]

Combining these bounds, we obtain
\[ a_{\nu,1} \geq \frac{4r^3d\theta^2 + \bO(h^3)}{32r^5d\theta^2+\bO(h^{9/2})} = \frac{1}{8r^2}\left(1+\bO(\sqrt{h})\right)\]
where $r^2 = \bO(h)$.
\end{proof}

\begin{lemma}[Bounds on coefficients (first derivatives)]\label{lem:bounds1}
There exists a constant $C>0$ such that for all sufficiently small $h>0$ and $\nu\in\R^2$, the coefficients defined by~\eqref{eq:coeffsGrad} satisfy
\[ \abs{b_{\nu,j}} \leq  \frac{C}{\sqrt{h}}.\]
\end{lemma}
\begin{proof}
We proceed as in the proof of Lemma~\ref{lem:bounds2} and compute the bounds
\[\begin{split}\det(A) = &(C_{\nu,3}S_{\nu,2}-C_{\nu,2}S_{\nu,3})(C_{\nu,1}^2S_{\nu,4}-C_{\nu,4}^2S_{\nu,1})\\&-(C_{\nu,1}S_{\nu,4}-C_{\nu,4}S_{\nu,1})(C_{\nu,3}^2S_{\nu,2}-C_{\nu,2}^2S_{\nu,3})\\
&\geq 8(r-2h)^5\left(1-\bO(d\theta^2)\right)^3\left(d\theta-\bO(d\theta^3)\right)^2\\
&= 8r^5d\theta^2 + \bO(h^4)
\end{split}\]
and
\[
0 \leq S_{\nu,4}(S_{\nu,3}C_{\nu,2}^2-S_{\nu,2}C_{\nu,3}^2) \leq 2r^4(2d\theta+\bO(d\theta^3))^2 = 8r^4d\theta^2 + \bO(h^4).
\]

Combining these, we find that
\[ 0 \leq b_{\nu,1} \leq \frac{8r^4d\theta^2+\bO(h^4)}{8r^5d\theta^2+\bO(h^4)}  = \frac{1}{r}\left(1+\bO(\sqrt{h})\right) \]
with $r = \mathcal{O}\left(\sqrt{h} \right)$.

The other coefficients are similar, though some are positive and some are negative.
\end{proof}

\subsection{Bounds on Lipschitz constants}
We next establish Lipschitz bounds on the functions $g_{1,\nu}, g_2$ defined in~\eqref{eq:gradFunctions1} and~\eqref{eq:gradFunctions2}, which play an important role in the discretization of functions of the gradient.

\begin{lemma}[Lipschitz bound on $g_{1,\nu}$]\label{lem:lipg1}
Let $x_i\in\G$ be fixed.  Then for $p\in\Tf_{x_i}$, the function
\[ g_{1,\nu}(p) = \Dt_{\nu\nu}\tilde{c}(x_i,T(x_i,p)).
 \]
has a Lipschitz constant $L$ satisfying
\[ L \leq L_{\tilde{c}} L_T + \bO(\sqrt{h}). \]
\end{lemma}
\begin{proof}
%Using consistency and regularity, we calculate
%\begin{align*}
%g_{1,\nu}(p)-g_{1,\nu}(q) &= \nu^TD_{xx}^2\left(\tilde{c}(x_i,T(x_i,p))-\tilde{c}(x_i,T(x_i,q))\right)\nu + \bO(\sqrt{h})\|p-q\| \\
  %&\leq \max\limits_{\abs{\nu}=1,x,y\in\Sf} \left\| \nabla_y \left(\nu^T D_{xx}^2\tilde{c}(x,y) \nu\right)\right\| \|T(x_i,p)-T(x_i,q)\| + \bO(\sqrt{h})\|p-q\| \\
	%&\leq \left(L_{\tilde{c}}L_T + \bO(\sqrt{h})\right)\|p-q\|. \qedhere
%\end{align*}
%
We first recall that $g_{1,\nu}(p)$ involves a finite difference discretization.  By consistency, we have
\[ g_{1,\nu}(p) =  \nu^T\left(D_{xx}^2\tilde{c}(x_i,T(x_i,p))\right)\nu + C(p)\sqrt{h}\]
where the coefficient $C(p)$ arising in the discretization error is at least Lipschitz continuous in $p$.  Then using regularity, we can calculate
\begin{align*}
g_{1,\nu}(p)-g_{1,\nu}(q) &= \nu^TD_{xx}^2\left(\tilde{c}(x_i,T(x_i,p))-\tilde{c}(x_i,T(x_i,q))\right)\nu + (C(p)-C(q))\sqrt{h} \\
  &\leq \max\limits_{\abs{\nu}=1,x,y\in\Sf} \left\| \nabla_y \left(\nu^T D_{xx}^2\tilde{c}(x,y) \nu\right)\right\| \|T(x_i,p)-T(x_i,q)\| + \bO(\sqrt{h})\|p-q\| \\
	&\leq \left(L_{\tilde{c}}L_T + \bO(\sqrt{h})\right)\|p-q\|. \qedhere
\end{align*}
\end{proof}

\begin{lemma}[Lipschitz bound on $g_{2}$]\label{lem:lipg2}
Let $x_i\in\G$ be fixed.  Then for $p\in\Tf_{x_i}$, the function
\[ g_{2}(p) = \frac{\abs{\det D^2_{xy}\tilde{c}(x_i,T(x_i,p))}}{f_2(T(x_i,p))}
 \]
has a Lipschitz constant $L$ satisfying
\[ L \leq \left\|\frac{1}{f_2}\right\|^2\left(\|f_2\| L_{H} + \| H\| L_{f_2}L_T\right). \]
\end{lemma}
\begin{proof}
Using the notation $H(p) = \abs{\det D^2_{xy}\tilde{c}(x_i,T(x_i,p))}$, we have
\begin{align*}
g_2(p)-g_2(q) &= \frac{H(p)}{f_2(T(x_i,p))} - \frac{H(q)}{f_2(T(x_i,q))}\\
  &= \frac{f_2(T(x_i,q))H(p)-f_2(T(x_i,p))H(q)}{f_2(T(x_i,p))f_2(T(x_i,q))}\\
	&= \frac{f_2(T(x_i,q))(H(p)-H(q)) + H(q)(f_2(T(x_i,q))-f_1(T(x_i,p)))}{f_2(T(x_i,p))f_2(T(x_i,q))}\\
	&\leq \left\|\frac{1}{f_2}\right\|^2\left(\|f_2\| L_{H} + \| H\| L_{f_2}L_T\right)\|p-q\|. \qedhere
\end{align*}
\end{proof}

\subsection{Lax-Friedrichs approximations}
We now verify that the Lax-Friedrichs type approximations for functions of the gradient defined in~\eqref{eq:discGrad} are both consistent and monotone.

\begin{lemma}[Consistency of functions of gradient]\label{lem:consistencyGrad}
Let $g$ be Lipschitz continuous with Lipschitz constant $L_g$ and $\phi\in C^2$.  Then 
\[\begin{split}
\tilde{g}^\pm(\Dt \phi(x_i)) = & g\left(\sum\limits_{\nu\in\{(1,0),(0,1)\}}\nu\sum\limits_{j=1}^4 b_{\nu,j}\left(\phi(x_{\nu,j})-\phi(x_i)\right)\right)\\ &\mp \epsilon_g\sum\limits_{\nu\in\{(1,0),(0,1)\}}\sum\limits_{j=1}^4a_{\nu,j}\left(\phi(x_{\nu,j})-\phi(x_i)\right).
\end{split}
\]
is a consistent approximation of $g(\nabla\phi)$.
\end{lemma}
\begin{proof}
We note that by construction, we have
\[ \lim\limits_{h\to0}\sum\limits_{\nu\in\{(1,0),(0,1)\}}\nu\sum\limits_{j=1}^4 b_{\nu,j}\left(\phi(x_{\nu,j})-\phi(x_i)\right) = \nabla\phi(x_i) \]
and 
\[ \lim\limits_{h\to0} \sum\limits_{\nu\in\{(1,0),(0,1)\}}\sum\limits_{j=1}^4a_{\nu,j}\left(\phi(x_{\nu,j})-\phi(x_i)\right) = \Delta\phi(x_i). \]

Using Lemmas~\ref{lem:bounds2}-\ref{lem:bounds1}, we can also bound $\epsilon_g$ via
\begin{align*} 
\epsilon_g &=\max\left\{\frac{L_g\abs{b_{\nu,j}}}{a_{\nu,j}} \mid j\in\{1,2,3,4\}, \, \nu \in \{(0,1),(1,0)\}\right\}\\
  &\leq \frac{L_g(C_b/\sqrt{h})}{C_a/h}\\
	&= C \sqrt{h}
\end{align*}
where the constant $C$ is independent of $h$.

Combining these results, we obtain
\[ \lim\limits_{h\to0}\tilde{g}^\pm(\Dt \phi(x_i)) = g(\nabla\phi(x_i)), \]
with a truncation error of $\bO(\sqrt{h})$.
\end{proof}

\begin{lemma}[Monotonicity of functions of the gradient]\label{lem:monotoneGrad}
Let $g$ be Lipschitz continuous with Lipschitz constant $L_g$ and $\phi\in C^2$.  Then the schemes $\tilde{g}^+(\Dt\phi(x_i))$ and  $\tilde{g}^-(\Dt\phi(x_i))$ are monotone and negative monotone respectively.
\end{lemma}
\begin{proof}
We verify monotonicity of $\tilde{g}^+(\Dt\phi(x_i))$; the other part of the argument is identical.

Denoting by $d_j$ the differences $\phi(x_i)-\phi(x_{\nu,j})$ allows us to express the scheme more compactly as
\[ G(d) = g\left(-\sum\limits_{\nu\in\{(1,0),(0,1)\}}\nu\sum\limits_{j=1}^4 b_{\nu,j}d_j\right)+ \epsilon_g\sum\limits_{\nu\in\{(1,0),(0,1)\}}\sum\limits_{j=1}^4a_{\nu,j}d_j. \]
We now introduce a perturbation $\delta>0$ into the argument $d_k$ to obtain
\begin{align*}
G(d + \delta \hat{d}_k)-G(d) &\geq -L_g\abs{b_{\nu,k}}\delta + \epsilon_ga_{\nu,k}\delta\\
  &\geq 0
\end{align*}
since $\epsilon_g \geq L\abs{b_{\nu,k}}/a_{\nu,k}$.  Therefore the scheme is monotone.
\end{proof}

\subsection{Consistency and monotonicity}

We now establish consistency and monotonicity of the approximation~\eqref{eq:discOTPDE}, which in turn establishes the convergence result (Theorem~\ref{thm:convergenceScheme}).

\begin{lemma}[Consistency]\label{lem:consistent}
Under the assumptions of Hypotheses~\ref{hyp:Smooth},\ref{hyp:grid}, the scheme $F^h$ defined by~\eqref{eq:discOTPDE} is consistent with the PDE~\eqref{eq:OTPDE} on the space of $C^2$ functions satisfying the constraint~\eqref{eq:cconvex}.
\end{lemma}
\begin{proof}
Let $\phi\in C^2$ satisfy the constraint~\eqref{eq:cconvex}.  Then the determinant in~\eqref{eq:OTPDE} can be equivalently expressed as in~\eqref{eq:detPDE}:
\[{\det}^+(D^2\phi(x) + A(x,\nabla \phi(x))) = \min_{\nu_1, \nu_2 \in V} \prod_{j=1}^2 \max \left\{ \frac{\partial^2 \phi(x)}{\partial \nu_j^2} + \left.\frac{\partial^2c(x,y)}{\partial\nu_j^2}\right|_{y = T(x,\nabla \phi(x))}, 0 \right\}.\]

By design and by Lemma~\ref{lem:consistencyGrad}, the components of the PDE are approximated consistently (with truncation error $\bO(\sqrt{h})$).  That is,
\[\lim\limits_{h\to0}\left\{\Dt_{\nu_j\nu_j}\phi(x_i)+\tilde{g}_{1,\nu_j}^-(\Dt \phi(x_i))\right\} = \frac{\partial^2 \phi(x)}{\partial \nu_j^2} + \left.\frac{\partial^2c(x,y)}{\partial\nu_j^2}\right|_{y = T(x,\nabla \phi(x))}\]
and
\[
\lim\limits_{h\to0}\tilde{g}_2^+\left(x_i,\Dt \phi(x_i)\right) = \frac{\abs{\det D^2_{xy}c(x_i,T(x_i,\nabla\phi(x_i)))}}{f_2(T(x_i,\nabla\phi(x_i)))}.
\]

We recall that the maximum and minimum operators are continuous, $f_1\in C^1$, and $\tilde{V}$ is a consistent approximation of the set $V$ with angular resolution $\bO(d\theta) = \bO(\sqrt{h})$.  Thus the combinations of these operators in the scheme $F^h$ satisfy
\begin{align*} \lim\limits_{h\to0,x_i\to x}-\min\limits_{(\nu_1,\nu_2)\in\tilde{V}}\prod\limits_{j=1}^2 \max\left\{ \Dt_{\nu_j\nu_j}\phi(x_i)+\tilde{g}_{1,\nu_j}^-(\Dt \phi(x_i)), 0\right\}+ f_1(x_i) \tilde{g}_2^+\left(x_i,\Dt \phi(x_i)\right) \\= -{\det}^+(D^2\phi(x) + A(x,\nabla \phi(x))) + {\abs{\det D^2_{xy}c(x,T(x,\nabla\phi(x)))}}\frac{f_1(x)}{f_2(T(x,\nabla\phi(x)))}, \end{align*}
which establishes consistency.
\end{proof}

\begin{corollary}[Truncation error]\label{cor:truncation}
Under the assumptions of Hypotheses~\ref{hyp:Smooth},\ref{hyp:grid}, the scheme $F^h$ defined by~\eqref{eq:discOTPDE} has local truncation error $\tau(h) = \mathcal{O}(\sqrt{h})$.
\end{corollary}

\begin{lemma}[Monotonicity]\label{lem:monotone}
Under the assumptions of Hypotheses~\ref{hyp:Smooth},\ref{hyp:grid}, the scheme $F^h$ defined by~\eqref{eq:discOTPDE} is monotone.
\end{lemma}
\begin{proof}
By construction (and see Lemma~\ref{lem:monotoneGrad}), the schemes for $\Dt_{\nu\nu}\phi$ and $\tilde{g}^-_{1,\nu}(\Dt u)$ are negative monotone.  Addition and the maximum function preserve this so that
\[ \max\left\{\Dt_{\nu\nu}u(x_i) + \tilde{g}^-_{1,\nu}(\Dt u(x_i)), 0\right\} \]
is also negative monotone for any $\nu\in\R^2$.  Since this is also non-negative, products of these terms preserve the negative monotonicity so that
\[ \min\limits_{(\nu_1,\nu_2)\in\tilde{V}}\prod\limits_{j=1}^2 \max\left\{ \Dt_{\nu_j\nu_j}u(x_i)+\tilde{g}_{1,\nu_j}^-(\Dt u(x_i)), 0\right\} \]
is also negative monotone.

We recall also that $f_1$ is non-negative and $\tilde{g}^+_2$ is monotone (Lemma~\ref{lem:monotoneGrad}).  Therefore the full scheme
\[-\min\limits_{(\nu_1,\nu_2)\in\tilde{V}}\prod\limits_{j=1}^2 \max\left\{ \Dt_{\nu_j\nu_j}u(x_i)+\tilde{g}_{1,\nu_j}^-(\Dt u(x_i)), 0\right\}+ f_1(x_i) \tilde{g}_2^+\left(x_i,\Dt u(x_i)\right)\]
is monotone.
\end{proof}

\subsection{Extensions to nonsmooth problems}\label{sec:nonsmooth}
The results of~\cite{Loeper_OTonSphere} ensure existence of weak solutions to the optimal transport PDE~\eqref{eq:OTPDE} in the relaxed setting where $f_1, f_2 \in L^1$ with $f_2$ bounded away from zero and $f_1$ bounded away from infinity.  
In~\cite{HT_OTonSphere}, the authors of this article showed that these solutions can be computed for the squared geodesic cost as in Theorem~\ref{thm:convergence} if the scheme $F^h$ is additionally required to underestimate when applied to the true solution. This result has now been extended to the regularized logarithmic cost via the remark~\ref{rmk:nonsmoothlog}.

Naturally underestimating schemes can be constructed for \MA type equations in some cases~\cite{BenamouDuval,HamfeldtBVP2}.  An alternate approach is to utilize a scheme of the form
\[ \tilde{F}^h(x,u(x)-u(\cdot)) = F^h(x,u(x)-u(\cdot)) - h^\alpha \]
for a sufficiently small $\alpha>0$.  This preserves the consistency and monotonicity of the original scheme, while decreasing the value of the scheme and forcing it to be negative when applied to the true solution of the PDE.
%\begin{equation}
%\begin{cases}
%F^h[u] + Ch^{\alpha} \leq 0 \\
%h^{\alpha} \rightarrow 0
%\end{cases}
%\end{equation}
%as $h \rightarrow 0$. In practical implementation, we can choose constant $C$ and an exponent $\alpha$ when we know, \textit{a priori} the solution to be non-smooth. Provided this can be done, then we have guaranteed convergence.

In addition, definition of consistency (Definition~\ref{def:consistency}) in terms of upper and lower semicontinuous envelopes of the PDE operator allows us to accommodate discontinuous data $f_1, f_2$ as in~\cite{HamfeldtBVP2}. This does not require any change in the way $f_1(x)$ is handled.  However, functions $f_2 \notin C^{0,1}$ must be carefully regularized to preserve consistency and monotonicity since it takes as its argument terms $T(x,\nabla u(x))$ that involve gradients.

The approach we propose is to introduce a discrete version of the target density function 
\bq\label{eq:f2h} f_2^h(y) =  (K_{h^{1/4}} * f_2)(y)\eq
where $K_{h^{1/4}}$ is a mollifier that ensures that the Lipschitz constant of $f_2^h$ satisfies
\[ L_{f_2^h} \leq h^{-1/4}. \]

From here, we use the same discretization of $f_2(T(x,\nabla u(x))$ introduced in~\eqref{eq:discGrad}.  
We note that in this case (following Lemma~\ref{lem:consistencyGrad}), the regularization parameter will satisfy
\[ \epsilon_{g_2^h} = \bO\left(L_{f_2^h}\sqrt{h}\right) = \bO(h^{1/4}). \]
Since this parameter converges to zero as $h\to0$, the resulting scheme will still be consistent in the sense of Definition~\ref{def:consistency}.  The monotonicity result (Lemma~\ref{lem:monotoneGrad}) is unchanged.

\section{Computational Results}\label{sec:implementation}
\subsection{Structured and Unstructured Grids}\label{grids}

The scheme we have built works well on both structured and unstructured grids, provided that they satisfy the mild conditions of Hypothesis~\ref{hyp:grid}. By structured we mean that there exists a deterministic way of building the grid. Likewise, unstructured here means there is a stochastic element in the construction of the grid.

Here, we describe four types of grids that satisfy these hypotheses and that we make use of in practice: the cube grid (structured), the random grid (fully unstructured), the latitude-longitude grid (unstructured), and the layered grid (structured).

The structured cube grid is constructed as follows. First, a grid of evenly-spaced points is generated on the faces of a cube which contains the sphere. Then, the points on such a grid on the cube are projected onto the sphere. See Figure~\ref{fig:cubegrid} for an example of the resulting grid.
%Due to the fact that the cube contains no cusps and since the cube is star-shaped, the resulting grid projected onto the sphere satisfies:
%
%\begin{equation}
%\mathcal{O}(h) \leq d_{\mathbb{S}^2}(x_i, x_j) \leq \mathcal{O}(h)
%\end{equation}
%for any points $x_i, x_j \in \mathbb{S}^2$. Thus, the hypotheses~\ref{hyp:grid} are satisfied. The figure~\ref{fig:cubegrid} shows what the cube grid looks like at different total number of points $N$:

The semi-unstructured latitude-longitude grid is constructed as follows. We begin with a structured grid composed of points equally spaced in both latitude $\theta$ and longitude $\phi$. However, this produces a highly over-resolved grid near the poles, which does not satisfy the required structure conditions.  In order to get rid of this redundancy, we stochastically remove grid points within a geodesic distance $\mathcal{O}(h)$ of both poles. That is, for each grid point $x_i = (\theta_i, \phi_i)$, we compute $\xi_i = \sin \theta_i$ and generate a value using the random variable $\Xi \sim \text{Unif}([0, 1])$. If $\Xi > \xi_i$, then we remove the point $(\theta_i, \phi_i)$ from the grid. The removal of these points creates a new unstructured grid that almost surely satisfies Hypothesis~\ref{hyp:grid}.  See Figure~\ref{fig:llgrid} for an example of such a grid.
%
%\begin{equation}
%\mathcal{O}(h) \leq d_{\mathbb{S}^2}(x_i, x_j) \leq \mathcal{O}(h)
%\end{equation}
%which shows that the hypotheses~\ref{hyp:grid} are satisfied, almost surely. The figure~\ref{fig:llgrid} shows what the latitude-longitude grid looks like for different total number of points $N$:

To construct the structured layered grid, we take an integer $n$ and $\epsilon = \mathcal{O}(h)$ and define the rows: $\text{row}_j = \epsilon + j\dfrac{\pi - 2 \epsilon}{n}$. For each row, we have $\text{col}_{kj} = j \phi_0 + k\dfrac{2\pi}{\floor{n*\sin \theta_n}}$, where $\phi_0 = \dfrac{1+\sqrt{5}}{2}$, which is the golden ratio. This creates a nice spread of points inspired by the seed packing of sunflowers. Then, we introduce the grid points $x_{jk} = (\theta, \phi) = (\text{row}_j, \text{col}_{jk})$ for $j = 1, \dots, n$ and $k = 1, \dots, \text{floor}\{n*\sin \theta_n\}$. This grid, by construction, will satisfy Hypothesis~\ref{hyp:grid}.  See Figure~\ref{fig:layeredgrid}.

\begin{figure}[htp]
	\subfigure[]{\includegraphics[width=0.45\textwidth]{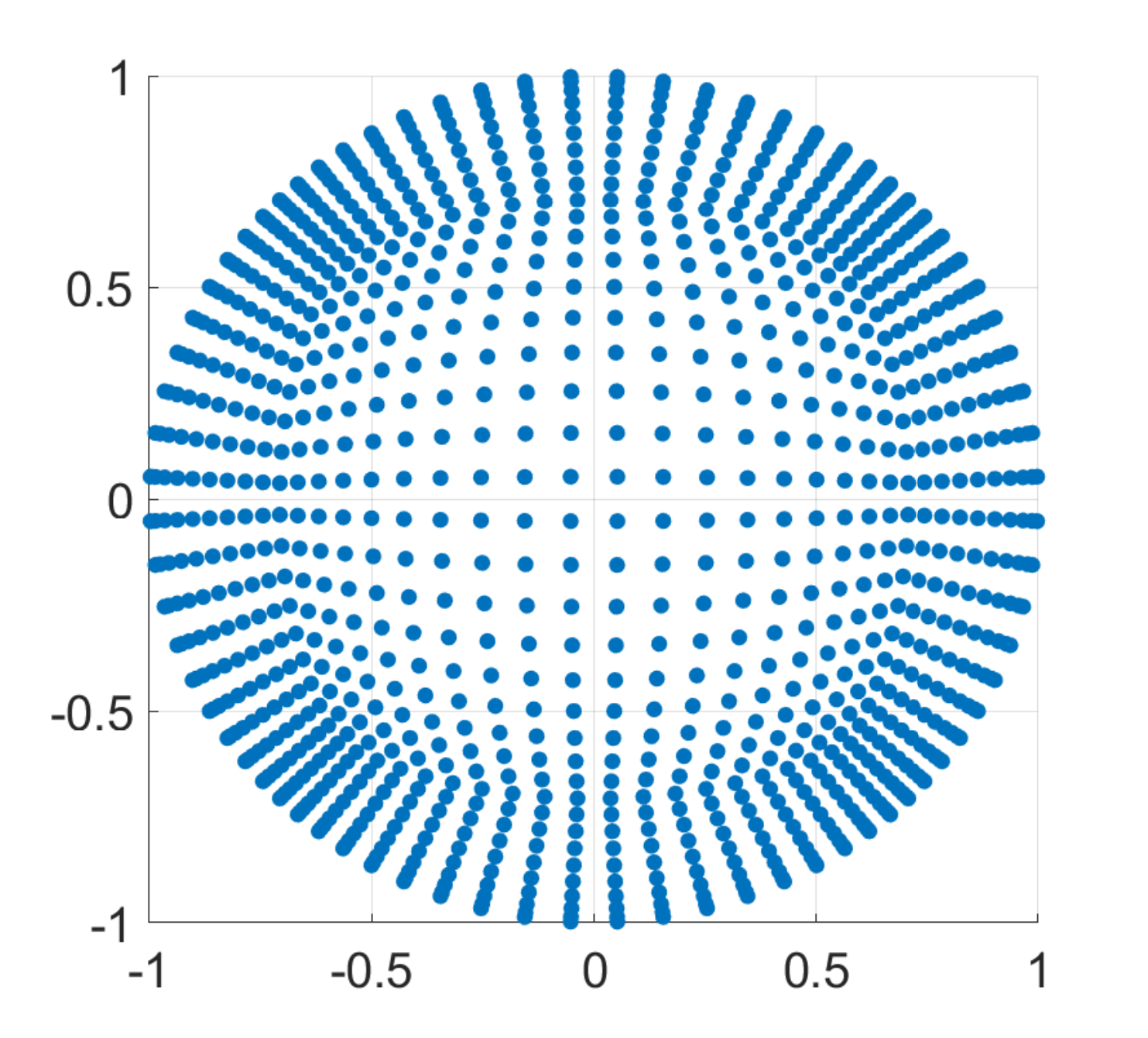}\label{fig:cubegrid}}
	\subfigure[]{\includegraphics[width=0.45\textwidth]{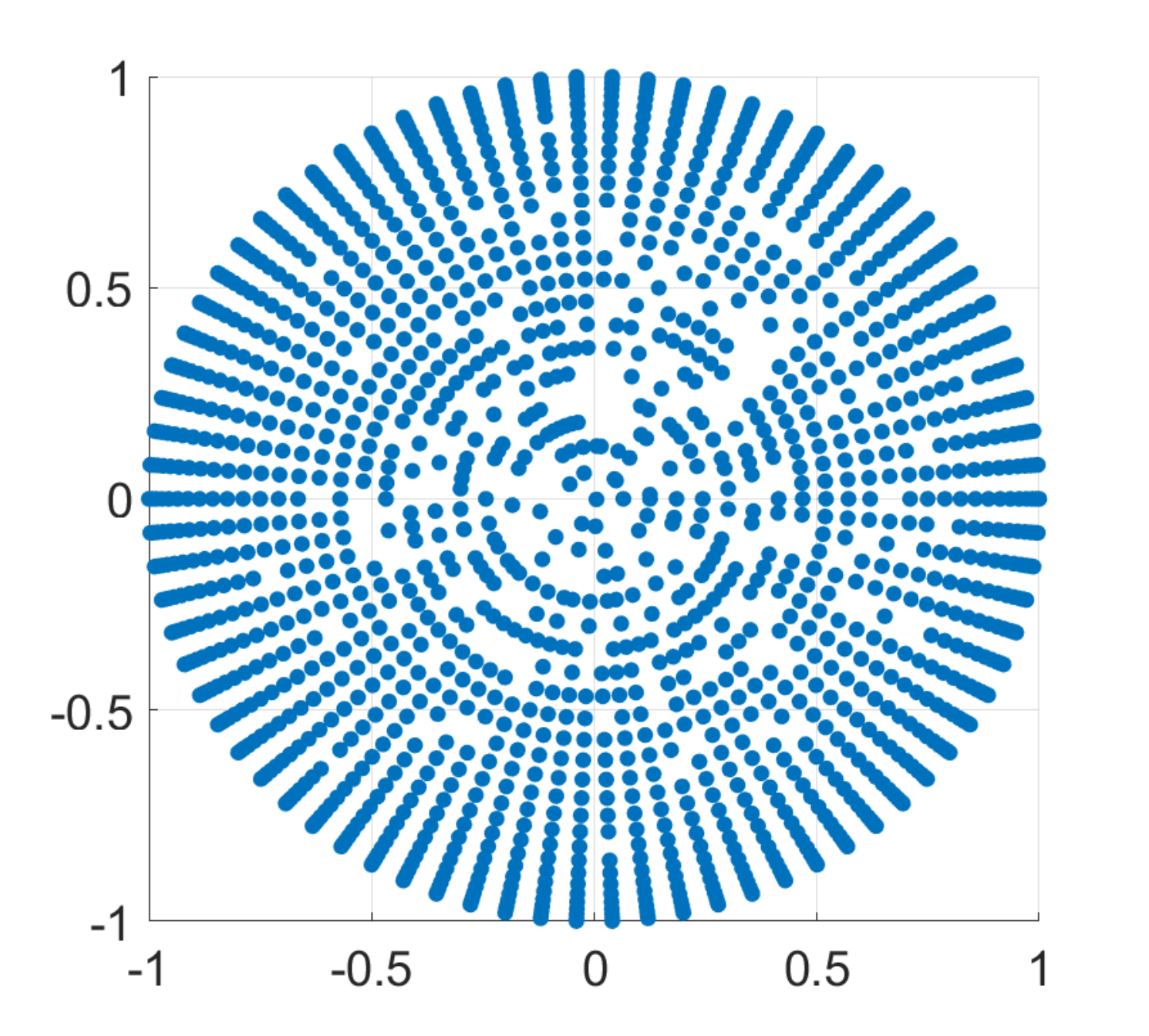}\label{fig:llgrid}}
	\subfigure[]{\includegraphics[width=0.45\textwidth]{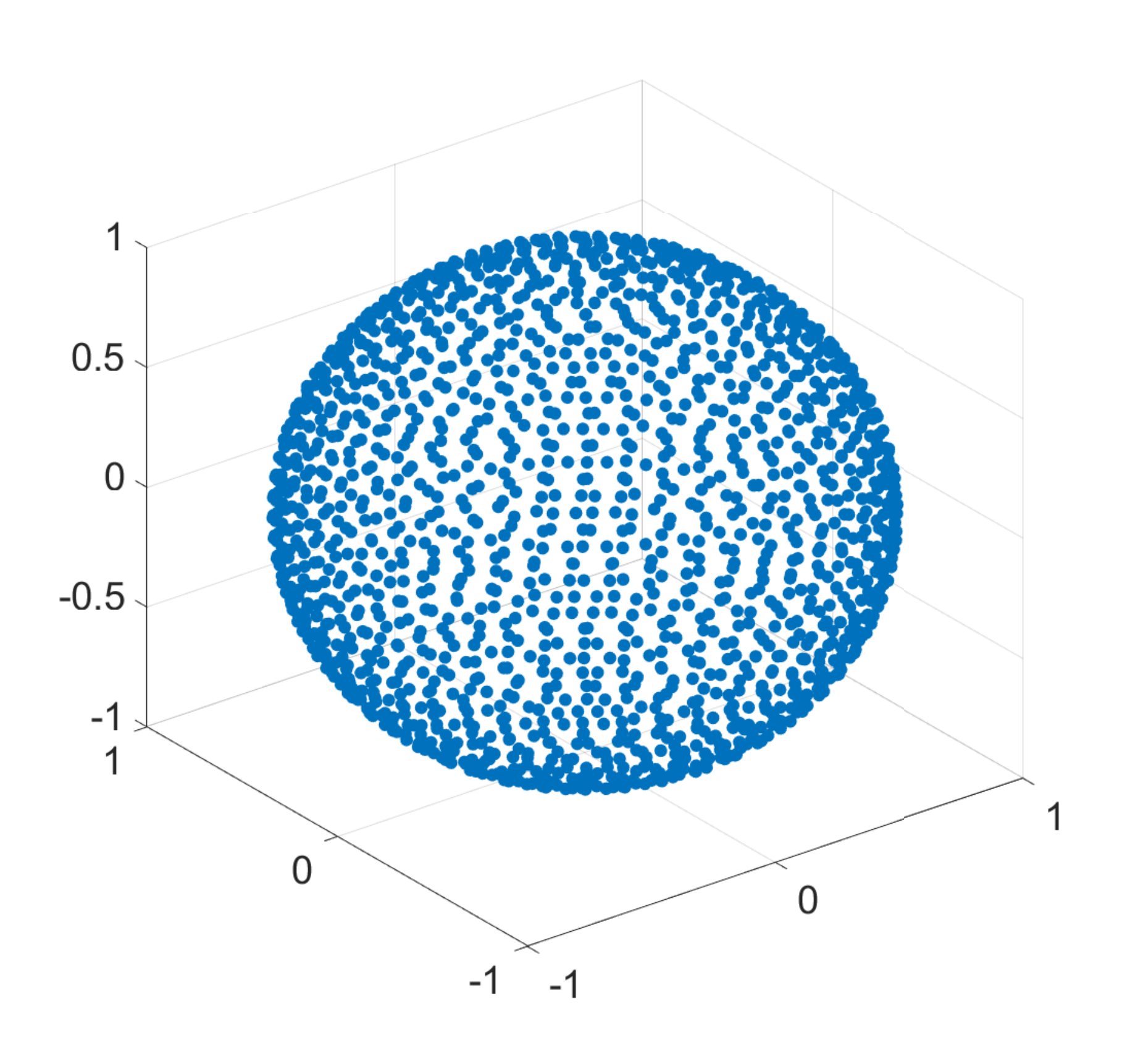}\label{fig:layeredgrid}}
	\subfigure[]{\includegraphics[width=0.45\textwidth]{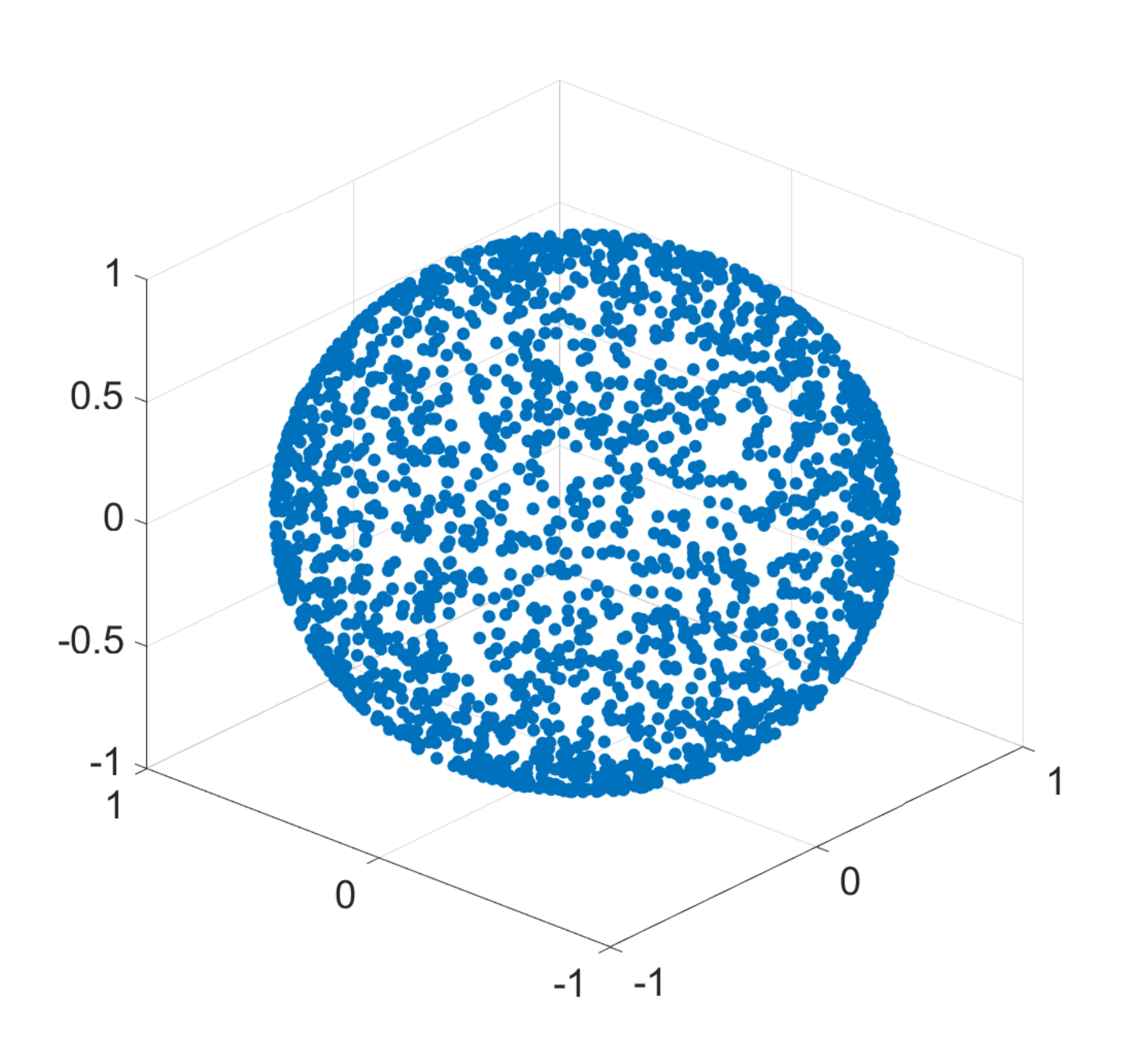}\label{fig:randomgrid}}
	\caption{Top down views of a \subref{fig:cubegrid}~cube grid and \subref{fig:llgrid} latitude-longitude grid.  Side views of a \subref{fig:layeredgrid}~layered grid and a \subref{fig:randomgrid}~randomgrid.}\label{fig:grids}
\end{figure}

%\begin{figure}[htp]
	%\subfigure[]{\includegraphics[height=5.6cm]{cubegrid1.png}}
	%\subfigure[]{\includegraphics[height=5.6cm]{cubegrid2.png}}
	%\caption{Cube grid for $N=488$ and $N=2168$ points respectively.}
	%\label{fig:cubegrid}
%\end{figure}

Finally, we consider a fully unstructured random grid. After defining the set
\[
R = \{(x, y) \in [0, \pi] \times [0, 1]: y \leq \sin(x) \}
\]
and the projection
\[
P(x, y) = x,
\]
we sample $\Phi \sim \text{Unif}([0, 2\pi])$ and $\tilde{\Theta} \sim \text{Unif}(R)$, then define $\Theta = P(\tilde{\Theta})$. 

We take as grid points the random variables \[(X,Y,Z) = (\sin \Theta \cos \Phi, \sin \Theta \sin \Phi, \cos \Theta)\], which satisfy the conditions that $(X, Y, Z) \sim \text{Unif}(\mathbb{S}^2)$. The resulting grid almost surely satisfies Hypothesis~\ref{hyp:grid}.  See Figure~\ref{fig:randomgrid} for an example of such a grid.

%\begin{figure}[htp]
	%\subfigure[]{\includegraphics[height=5.6cm]{randomgrid1.png}}
	%\subfigure[]{\includegraphics[height=5.6cm]{randomgrid2.png}}
	%\caption{Random grid for $N=488$ and $N=2168$ points, resp.}
	%\label{fig:randomgrid}
%\end{figure}

%\begin{figure}[htp]
	%\subfigure[]{\includegraphics[height=5.6cm]{llgrid1.png}}
	%\subfigure[]{\includegraphics[height=5.6cm]{llgrid2.png}}
	%\caption{Latitude-Longitude grid for $N=840$ and $N=2601$ points, resp.}
	%\label{fig:llgrid}
%\end{figure}
%

%\begin{figure}[htp]
	%\includegraphics[height=10cm]{rotationgrid.png}
	%\caption{Layered grid with $N=1600$ points}
	%\label{fig:layeredgrid}
%\end{figure}

\subsection{Recovering constant solutions}

For both cost functions, in the case where $f_1 = f_2$, the resulting solution will be a constant function ($u(x)=0$). Figure~\ref{fig:constants} shows the solutions obtained for both the squared geodesic and logarithmic cost functions.  Importantly, these are effectively constant to within a tolerance less than the expected consistency error of the method.

\begin{figure}[htp]
	\subfigure[]{\includegraphics[width=0.45\textwidth]{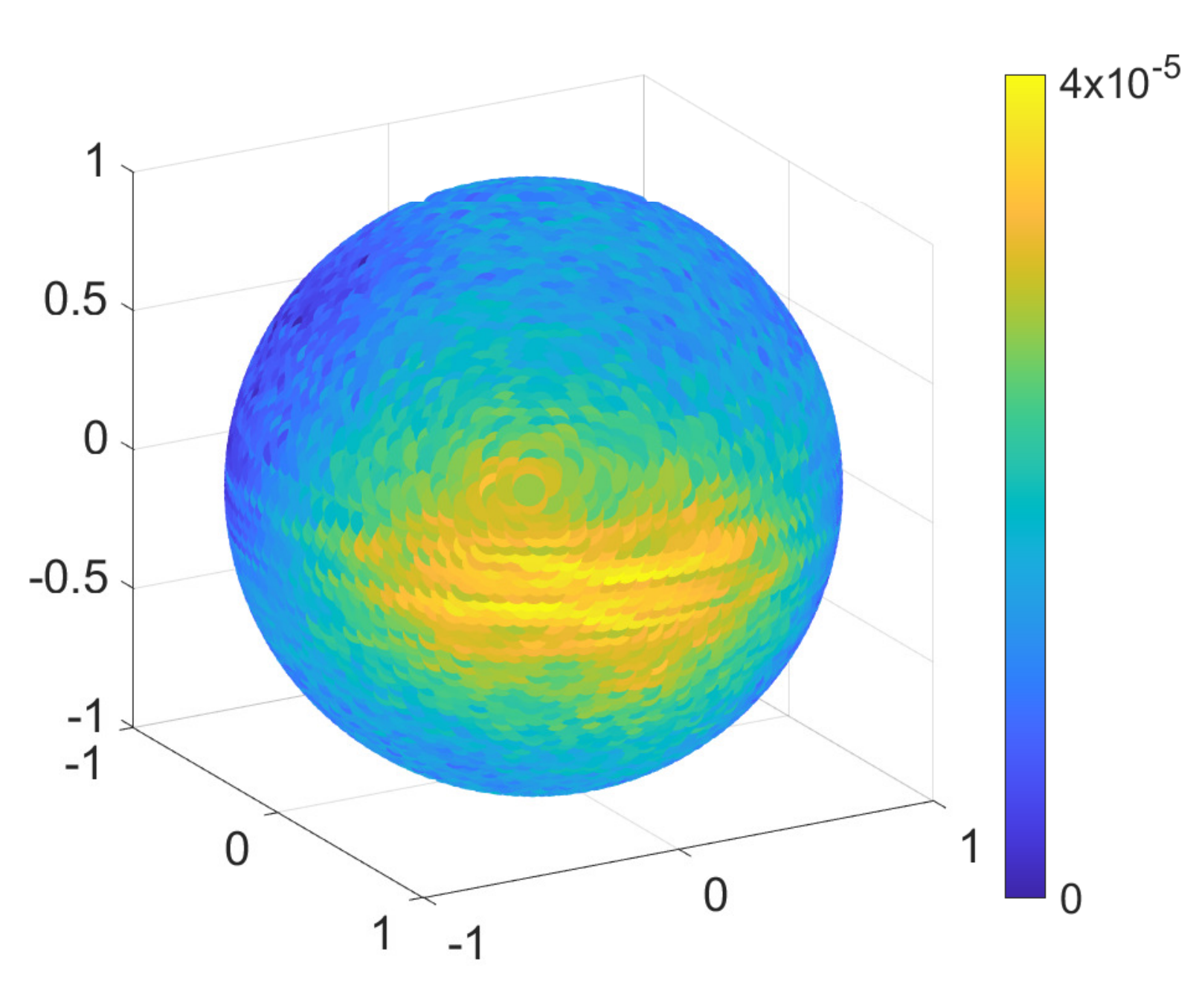}\label{fig:constantd2}}
	\subfigure[]{\includegraphics[width=0.45\textwidth]{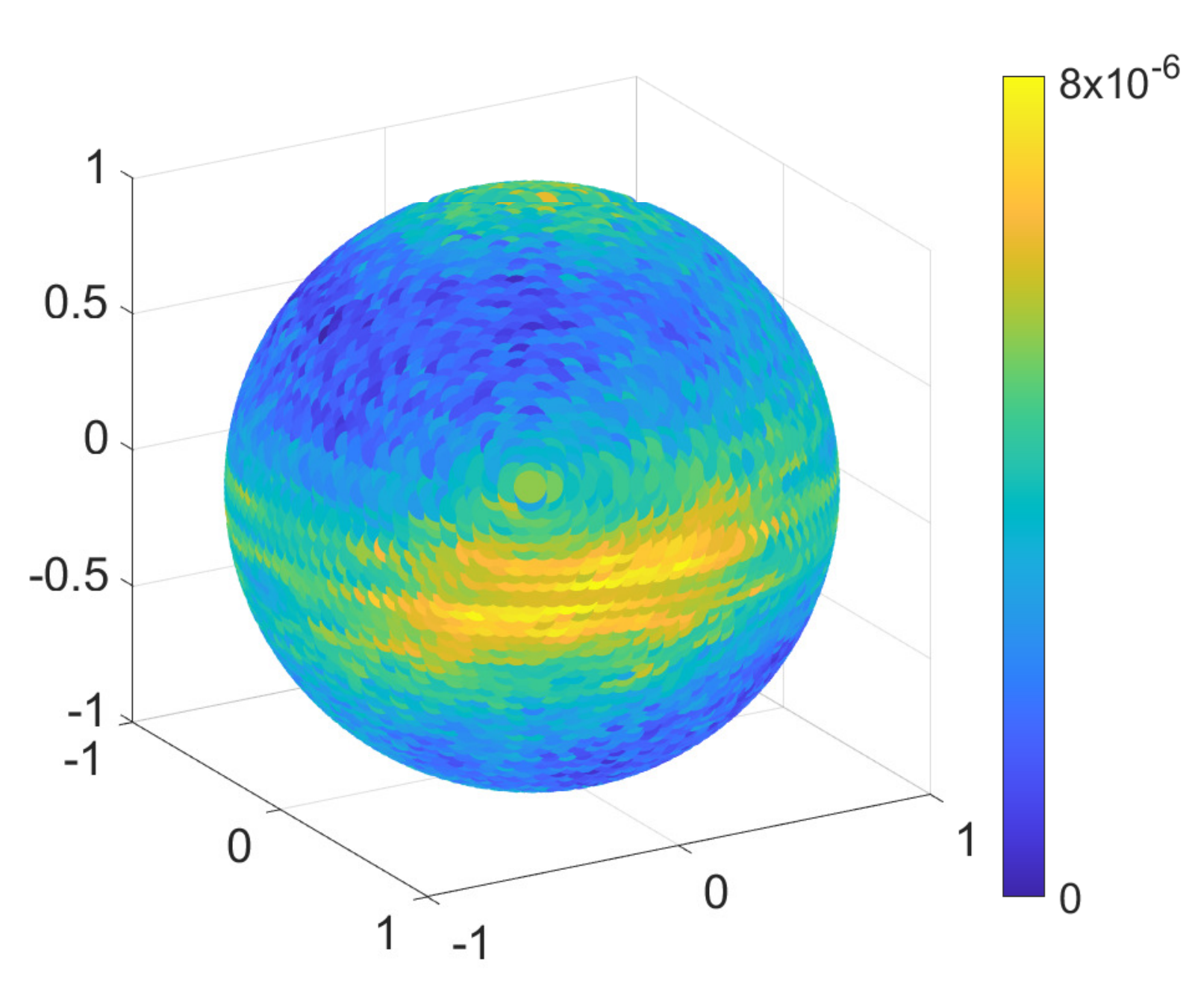}\label{fig:constantlog}}
	\caption{Solutions obtained for the \subref{fig:constantd2}~squared geodesic cost and \subref{fig:constantlog}~logarithmic cost when $f_1=f_2$ on a layered grid consisting of $N=7722$ points.}\label{fig:constants}
\end{figure}

%\begin{figure}[htp]
	%\subfigure[]{\includegraphics[height=5.4cm]{constant1.png}}
	%\subfigure[]{\includegraphics[height=5.1cm]{constantres1.png}}
	%\caption{Constant solution for $f_1 = f_2$ using $N=7722$-point layered grid, squared geodesic cost. Residual plot on the right shows the performance of the parabolic solver over $\sim 1000$ iterations.}
	%\label{fig:identitysg}
%\end{figure}
%
%And the figure~\ref{fig:identitylog} demonstrates the constant solution for the logarithmic cost.
%
%\begin{figure}[htp]
	%\subfigure[]{\includegraphics[height=5.6cm]{constant2.png}}
	%\subfigure[]{\includegraphics[height=5.1cm]{constantres2.png}}
	%\caption{Constant solution for $f_1 = f_2$ using $N=7722$-point layered grid, logarithmic cost. Residual plot on the right shows the performance of the parabolic solver over $\sim 1500$ iterations.}
	%\label{fig:identitylog}
%\end{figure}

\subsection{Small Perturbation}
Here we demonstrate with computations for the squared geodesic cost what happens when the target mass density $f_2$ is obtained through a slight perturbation (a rotation through an angle $\theta$) of the source mass density $f_1$.
In particular, we choose the density functions
\begin{equation}
\begin{cases}
f_1(x,y,z) = \frac{1}{4\pi - 4} \left( 1 - 0.5 \cos \left( \frac{\pi}{2} x \right) \right) \\
f_2(x,y,z) = \frac{1}{4\pi - 4} \left( 1 - 0.5 \cos \left( \frac{\pi}{2} (x \cos \theta + y \sin \theta) \right) \right)
\end{cases}
\end{equation}

This problem has the flavor of a translation.  In the Euclidean setting, translations are exact solutions of the optimal transport problem. However, it is not the case that rotations are exact solutions on the sphere.  See Figure~\ref{fig:gradrot} for a top-view of the computed gradient map.  In particular, we observe that the bulk of the mass does undergo a clockwise rotation.  However, in order to conserve mass and regularity, there is also a backwards ``flow'' observed in areas of low density (top and bottom of the figure).
%The figure~\ref{fig:fandgrot} shows the source and target masses with $N=2168$ points.

%\begin{figure}[htp]
%	\includegraphics[height=9cm]{fandgrot.png}
%	\caption{Target mass $f_2$ (right) slightly rotated $\theta = 0.1$ with respect to source mass $f_1$ (left) on a $2168$-point cube grid.}
%	\label{fig:fandgrot}
%\end{figure}

\begin{figure}[htp]
	\includegraphics[width=0.8\textwidth]{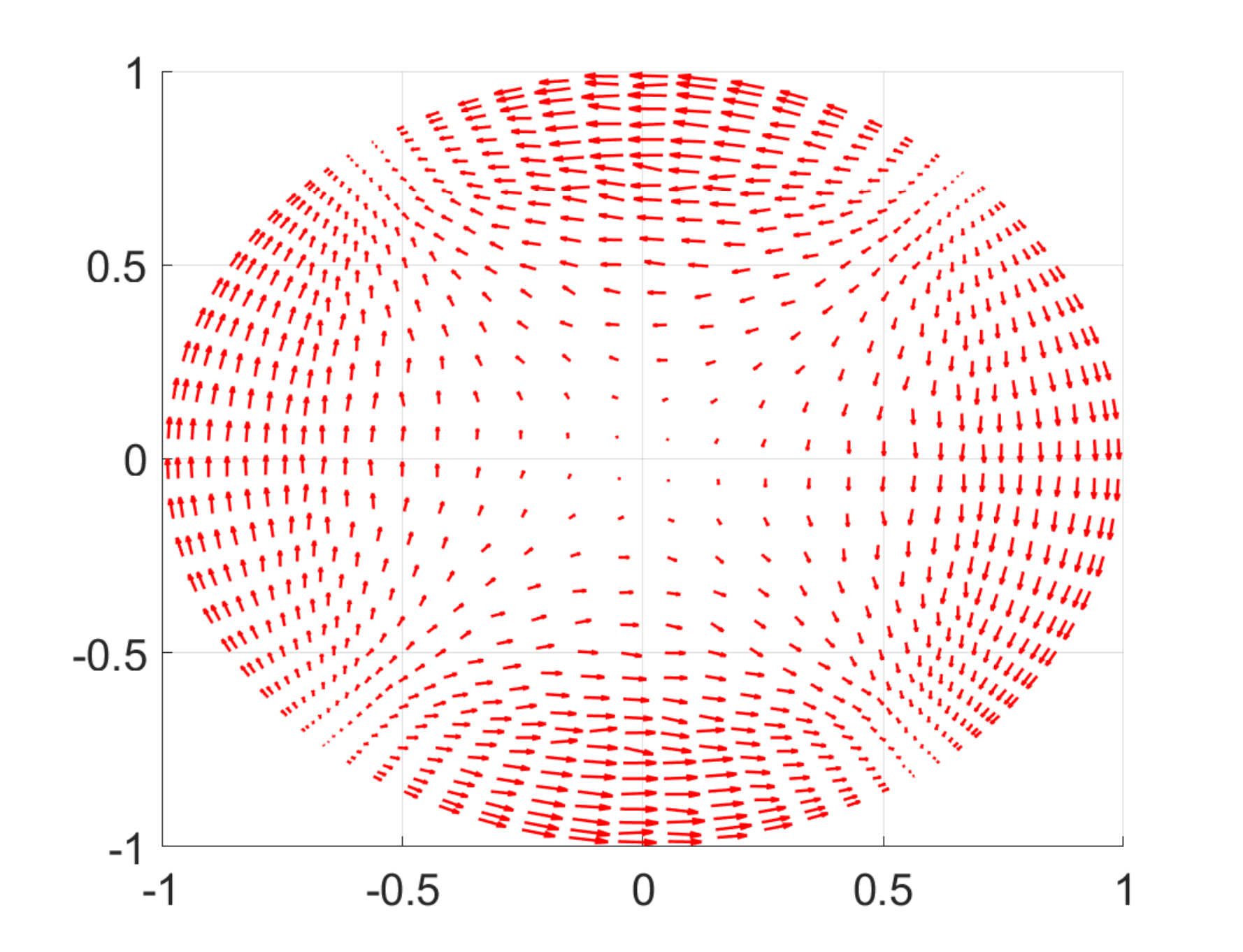}
	\caption{Top view of the local gradient obtained numerically when $f_2$ is obtained from $f_1$ through a small rotation. The solution was computed using a cube grid with $N=2168$ points. }
	\label{fig:gradrot}
\end{figure}

%\begin{figure}[H]
%	\includegraphics[height=9cm]{urot.png}
%	\caption{Corresponding solution for the rotated mass on sphere using $2168$-point cube grid.}
%\end{figure}

\subsection{Comparing structured and unstructured grids}
 To show the robustness of our generalized finite difference scheme with respect to the structure of the grid, we next show a side-by-side comparison of solutions obtained using a fully structured layered grid and a fully unstructured random grid. We choose a non-smooth source density $f_1$ and constant target density $f_2$:
\bq\label{eq:compare}
\begin{cases}
f_1(\theta, \phi) =  (1-\epsilon) \frac{1}{5.8735} \left(\theta-\frac{\pi}{2}\right)^2 + \frac{\epsilon}{4\pi}\\
f_2(\theta, \phi) = \frac{1}{4\pi}
\end{cases}
\eq
for $\epsilon=0.1$. See Figure~\ref{fig:compare} for the computed solutions, which are identical to within a tolerance on the order of the computed residual.

\begin{figure}[htp]
	\subfigure[]{\includegraphics[height=5.2cm]{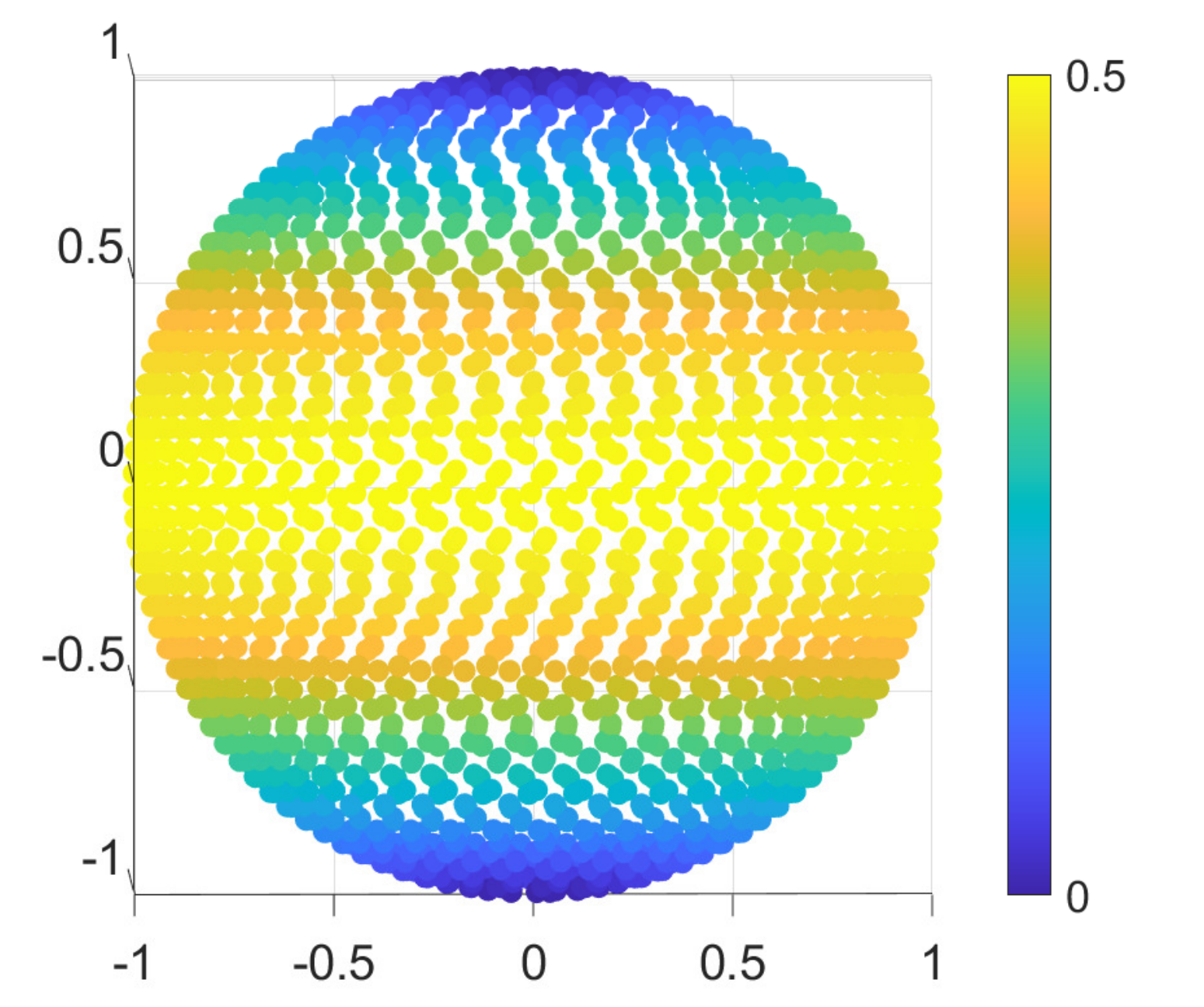}\label{fig:structured}}
	\subfigure[]{\includegraphics[height=5.2cm]{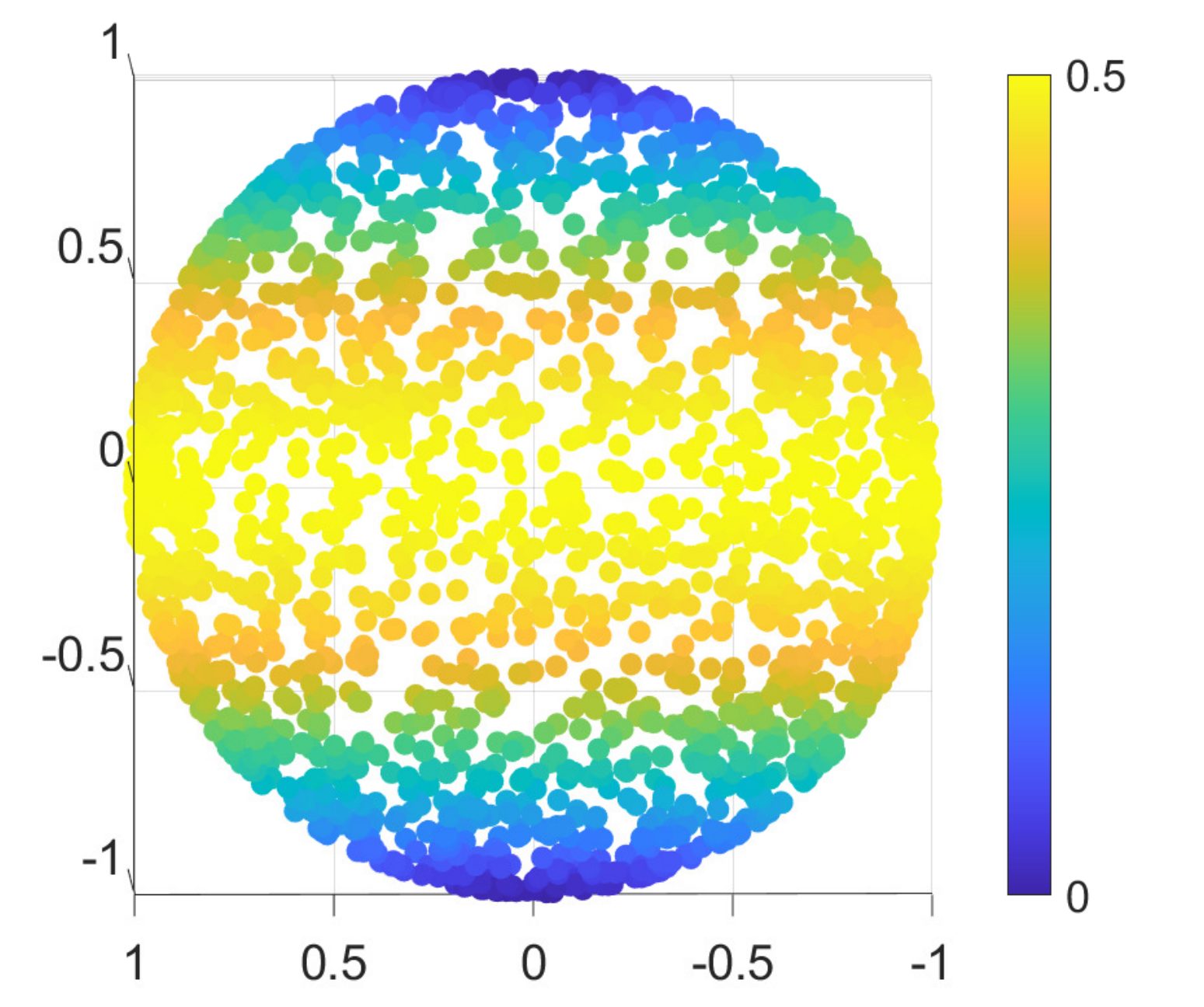}\label{fig:unstructured}}
	\caption{Solution $u$ obtained for densities~\eqref{eq:compare} on $N=2006$ point \subref{fig:structured}~layered and \subref{fig:unstructured}~random grids.}
	\label{fig:compare}
\end{figure}

\subsection{Mesh generation}
In the context of mesh generation, one of the most important aspects of the squared geodesic cost is the lack of tangling in mesh restructuring.
%in moving-mesh methods is the lack of tangling in the mesh restructuring, due to the regularity imparted by the squared geodesic cost function and the fact that the mapping is an invertible function of the gradient of the $c$-convex function $u$: $T(x) = \text{exp}_{x} \left( \nabla u(x) \right)$.
To demonstrate that our method preserves this critical property, we construct an explicit mesh and preserve the edge connections under the computed transport map $T$.  We begin with a structured cube grid, where edges are obtained from the edges of the original cube grid that was projected onto the sphere.  We select density functions with the goal of producing a transport map $T$ that will concentrate mesh points around the equator.
%In order to demonstrate moving mesh methods, we construct an explicit mesh and keep track of the edges (see~\ref{fig:cubemesh}). The following figures use the cube grid, where the edges are connected in the logical way. The figure~\ref{fig:cubemesh} shows the mesh before and after being projected onto the sphere.
%
%\begin{figure}[htp]
	%\subfigure[]{\includegraphics[height=6.7cm]{cubemesh1.png}}
	%\subfigure[]{\includegraphics[height=6.7cm]{cubemesh2.png}}
	%\caption{Example cube mesh with vertices and edges on cube (left) and projected onto the sphere (right), for $N=488$ vertices}
	%\label{fig:cubemesh}
%\end{figure}
%
%We then take this cube grid projected onto the sphere and perform a mesh restructuring using the squared geodesic cost function. The mass densities are chosen so that the resulting map $T$ will concentrate more points around the equator. After performing our computations for the squared geodesic cost, we do not observe any mesh tangling in the resulting figure~\ref{fig:movingmesh}. The source and target mass density functions are chosen to be:
%
\begin{equation}\label{eq:meshgen}
\begin{cases}
f_1(\theta, \phi) = \frac{1-\epsilon}{5.8735} \left(\theta-\frac{\pi}{2} \right)^2 + \frac{\epsilon}{4\pi} \\
f_2(\theta, \phi) = \frac{1}{4\pi}
\end{cases}
\end{equation}
where $\epsilon=0.1$.

The resulting mesh restructuring is pictured in Figure~\ref{fig:movingmesh}.  As desired, the final mesh concentrates grid points around the equator without introducing any tangling into the mesh.

\begin{figure}[htp]
	\subfigure[]{\includegraphics[width=\textwidth,clip=true,trim={0.6in 1.2in 0.6in 1.2in}]{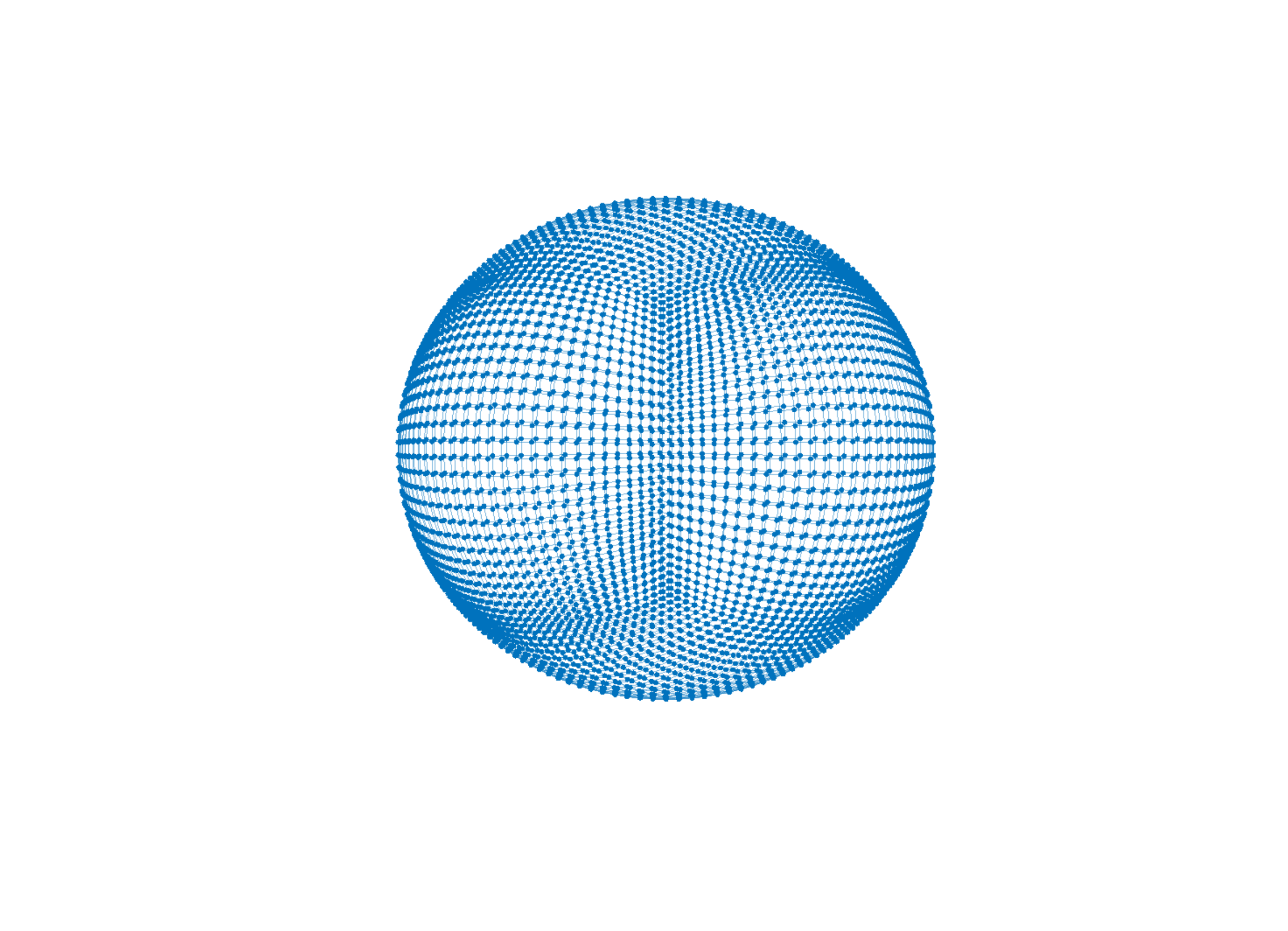}\label{fig:mesh1}}
	\subfigure[]{\includegraphics[width=\textwidth,clip=true,trim={0.5in 1in 0.5in 1in}]{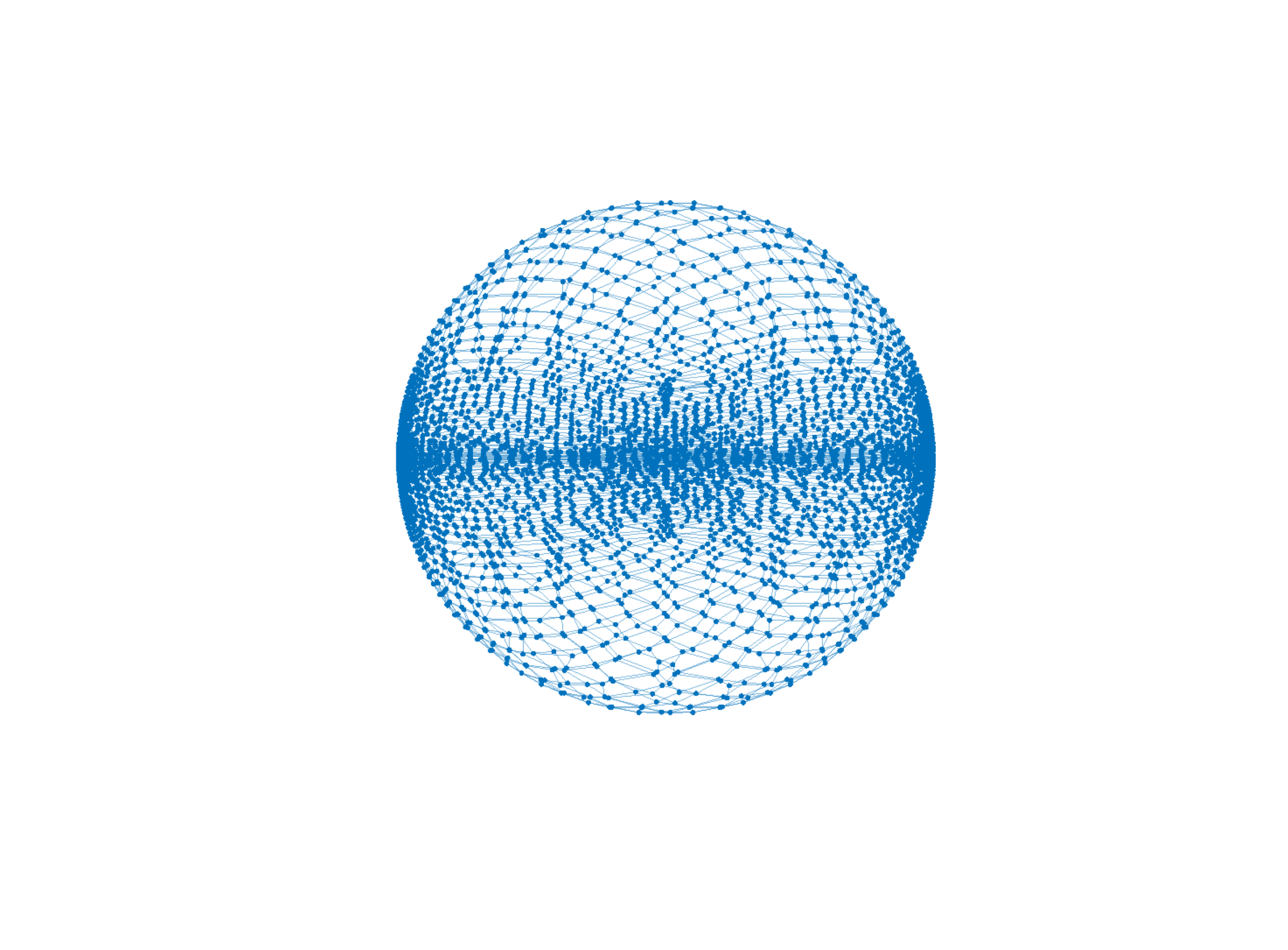}\label{fig:mesh2}}
	\caption{\subref{fig:mesh1}~The original $N=5048$ point cube mesh and \subref{fig:mesh2}~the mesh obtained by transporting through the computed optimal transport map.}
	\label{fig:movingmesh}
\end{figure}

\subsection{Logarithmic cost with full support}
For the reflector antenna problem, previous available methods have successfully performed computations when the mass densities $f_1$ and $f_2$ have support on a subset of $\Sf$; see~\cite{RomijnSphere}. This does not allow the reflector antenna problem to be solved for cases where the source light intensity is, for example, omnidirectional. Here we demonstrate the computation for the case where both the mass densities have full support. 
%\begin{equation}
%\begin{cases}
%f_1(x, y, z) = \frac{1-\epsilon}{5.8735} \left(\arccos(x)-\frac{\pi}{2} \right)^2 + \frac{\epsilon}{4\pi} \\
%f_2(z, y, z) = \frac{1}{4\pi}
%\end{cases}
%\end{equation}
%where $\epsilon=0.1$.
%
%The computation is performed easily using our method, with results pictured in Figure~\ref{fig:fullsupport}. 
%
%\begin{figure}[htp]
	%\subfigure[]{\includegraphics[height=5.7cm]{fandglog1}}
	%\subfigure[]{\includegraphics[height=5.7cm]{log1}}
	%\caption{Source and target masses (left) for full support computation with logarithmic cost. Solution $u$ (right) for logarithmic cost performed on $N=7722$-point layered grid.}
	%\label{fig:fullsupport}
%\end{figure}
%
%\begin{figure}[htp]
	%\includegraphics[height=10cm]{gradlog1}
	%\caption{Northern hemisphere gradient plot for source and target densities with full support with logarithmic cost on $N=2006$-point layered grid.}
	%\label{fig:gradlog1}
%\end{figure}

Since applications in optics can often involve highly non-smooth densities, we take as our source density the image of a world map, which is mapped to a constant target density.  Despite the potential singularity in the cost function and the non-smoothness of the problem data and solution, our method handles this problem with ease.  See Figure~\ref{fig:globe} for a visualization of the density functions and computed (non-smooth) solution.  In particular, the outlines of Africa, the Middle East, and Asia are easily discerned in the solution.
%The figure~\ref{fig:fandgglobe} shows the latitude-longitude grid being used to compute a non-smooth map of the world to a constant target density. The first figure shows the source and target masses:
%
%\begin{figure}[htp]
	%\includegraphics[height=10cm]{fandgglobe}
	%\caption{Mapping the globe to a constant function}
	%\label{fig:fandgglobe}
%\end{figure}
%
%The resulting solution of the PDE is shown in figure~\ref{fig:globe}.

\begin{figure}[htp]
	\subfigure[]{\includegraphics[width=0.8\textwidth]{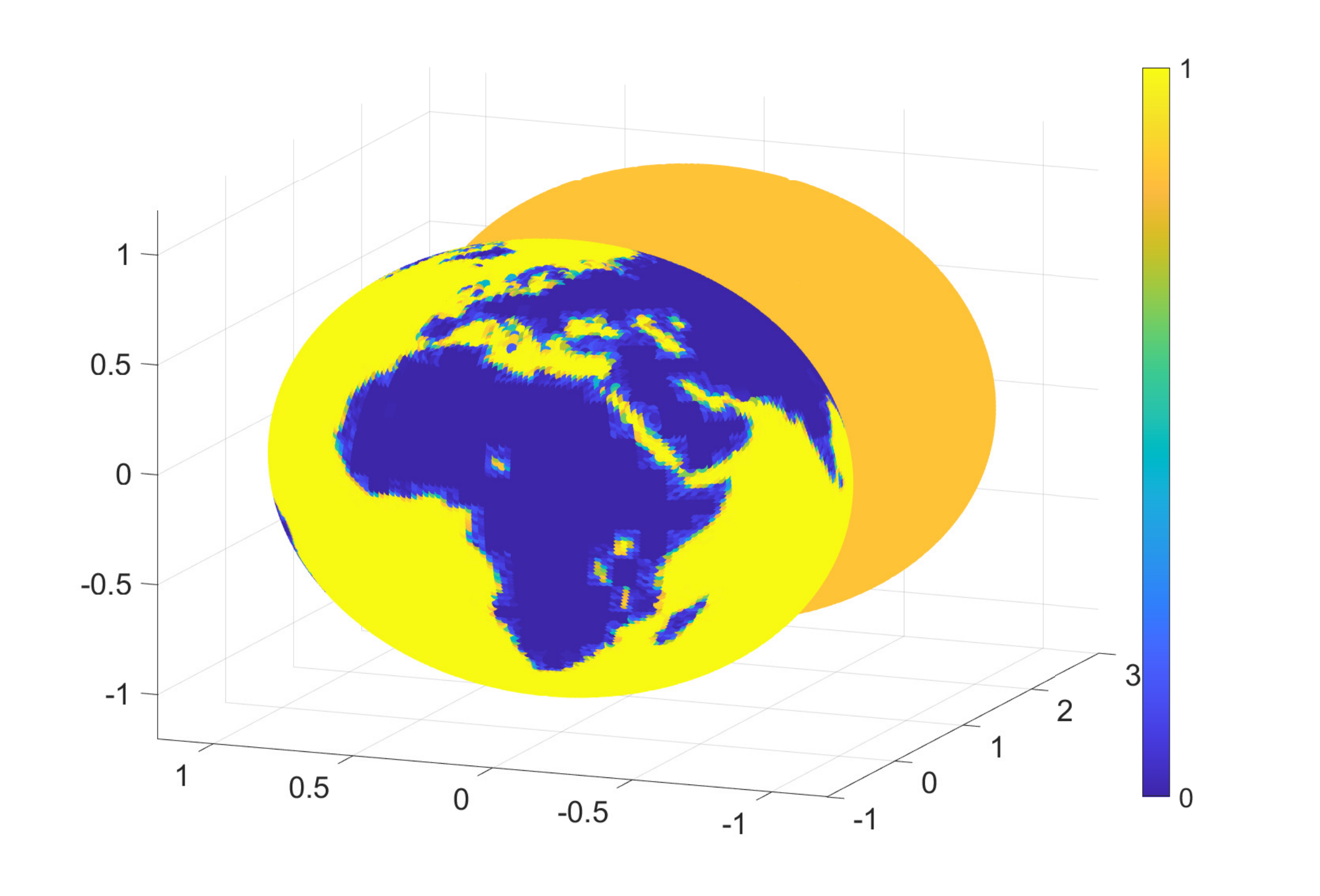}\label{fig:globeDensities}}
	\subfigure[]{\includegraphics[width=0.8\textwidth]{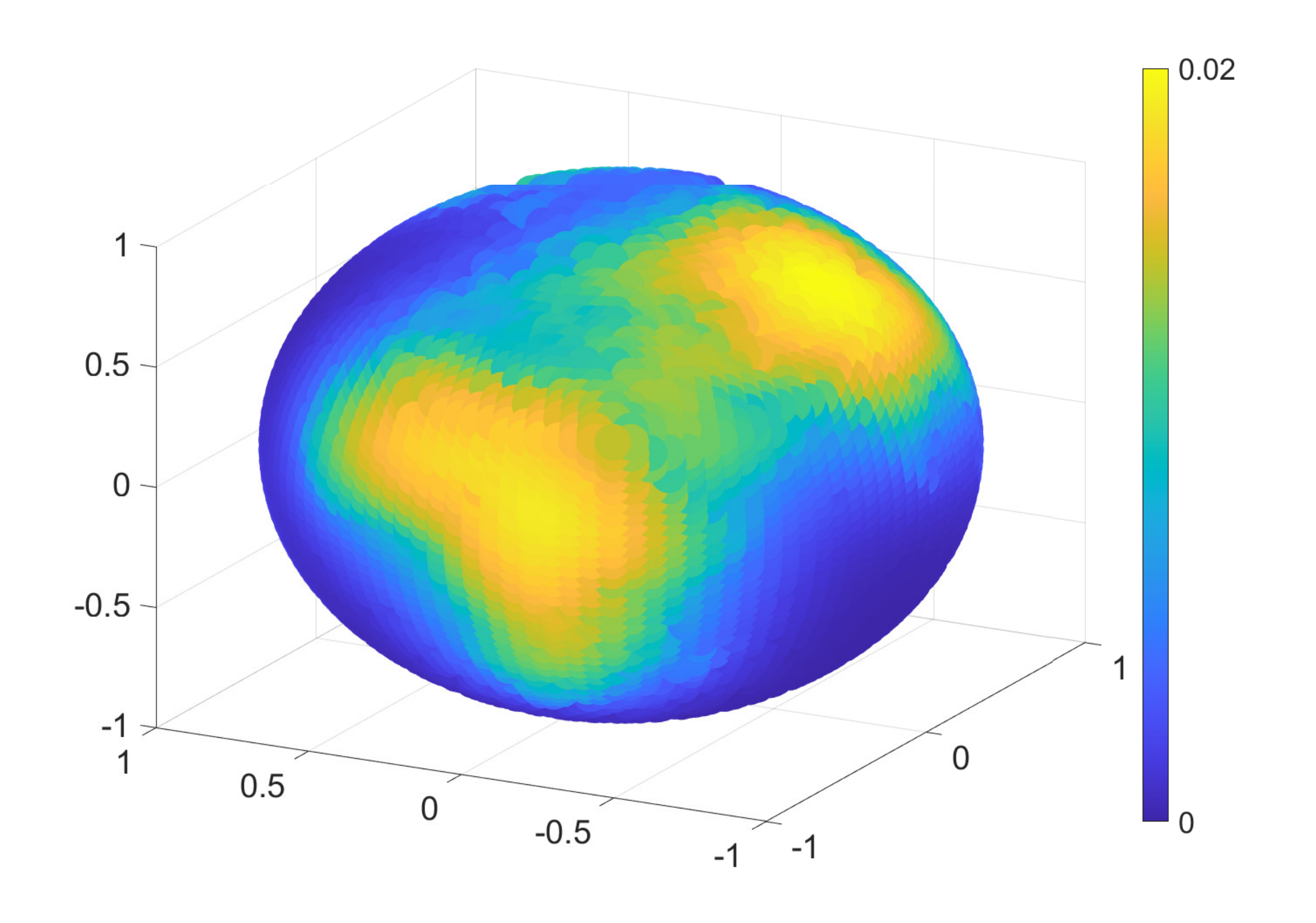}\label{fig:globeSolution}}
	\caption{\subref{fig:globeDensities}~Density functions for the world map given on a $N=5098$ point latitude-longitude grid and \subref{fig:globeSolution}~computed solution $u$ using the logarithmic cost.}
	\label{fig:globe}
\end{figure}

\subsection{Nonsmooth Examples}\label{nonsmoothexamples}

%We demonstrate computations when $f_1$ is unbounded, with the squared geodesic cost. Let the source and target masses be:

%\begin{equation}
%\begin{cases}
%f_1(\theta,\phi) = \frac{1}{2\pi \cdot 1.86691}\theta^{-1/4} \\
%f_2(\theta, \phi) = 1/4*\pi
%\end{cases}
%\end{equation}

%The source and target masses and the solution $u$ are shown in figure~\ref{fig:unbounded}.

%\begin{figure}[htp]
%	\subfigure[]{\includegraphics[height=5.5cm]{fandgunbounded}}
%	\subfigure[]{\includegraphics[height=5.5cm]{unbounded.png}}
%	\caption{Source and target masses. Source mass unbounded as $\theta \rightarrow 0$ on a $7722$-point layered grid. The solution $u$ is on the right.}
%	\label{fig:unbounded}
%\end{figure}

%We demonstrate a computation for a non-Lipschitz $f_2$ with the squared geodesic cost, in figure~\ref{fig:unboundedlip}.

%\begin{equation}
%\begin{cases}
%f_1(\theta, \phi) = 1/4\pi \\
%f_2(\theta, \phi) = (1-\epsilon)\frac{\theta^{3/4}}{17.2747} + \frac{\epsilon}{4\pi}
%\end{cases}
%\end{equation}
%where $\epsilon = 0.5$. The resulting plots are shown below:

%\begin{figure}[htp]
%	\subfigure[]{\includegraphics[height=5.5cm]{fandgnonlip}}
%	\subfigure[]{\includegraphics[height=5.5cm]{nonlip}}
%	\caption{Source and target masses. Target mass unbounded Lipschitz constant as $\theta \rightarrow 0$ on a $7722$-point layered grid. Solution $u$ is on the right.}
%	\label{fig:unboundedlip}
%\end{figure}

Finally, we present the results of a computation (using the squared geodesic cost) where $f_1$ is unbounded and $f_2$ is not Lipschitz. Recall that this is a situation where the solution $u$ is only guaranteed to be $C^1(\mathbb{S}^2)$ and the non-Lipschitz property of $f_2$ can easily cause issues regarding monotonicity and consistency. However, these issues can be resolved using the ideas in~\autoref{sec:nonsmooth}. The density functions are given by
\begin{equation}
\begin{cases}
f_1(\theta, \phi) =  \frac{1}{2\pi \cdot 1.86691}\theta^{-1/4} \\
f_2(\theta, \phi) = (1-\epsilon)\frac{\theta^{3/4}}{17.2747} + \frac{\epsilon}{4\pi}.
\end{cases}
\end{equation}
where $\epsilon = 0.5$. 

Despite the very strong singularities present in this example, our numerical method has no difficulty computing a solution.  The density functions and computed gradient are shown in Figure~\ref{fig:nonsmooth}.  As expected, we observe mass being transported downard away from the singularity.

\begin{figure}[htp]
	\subfigure[]{\includegraphics[width=0.8\textwidth]{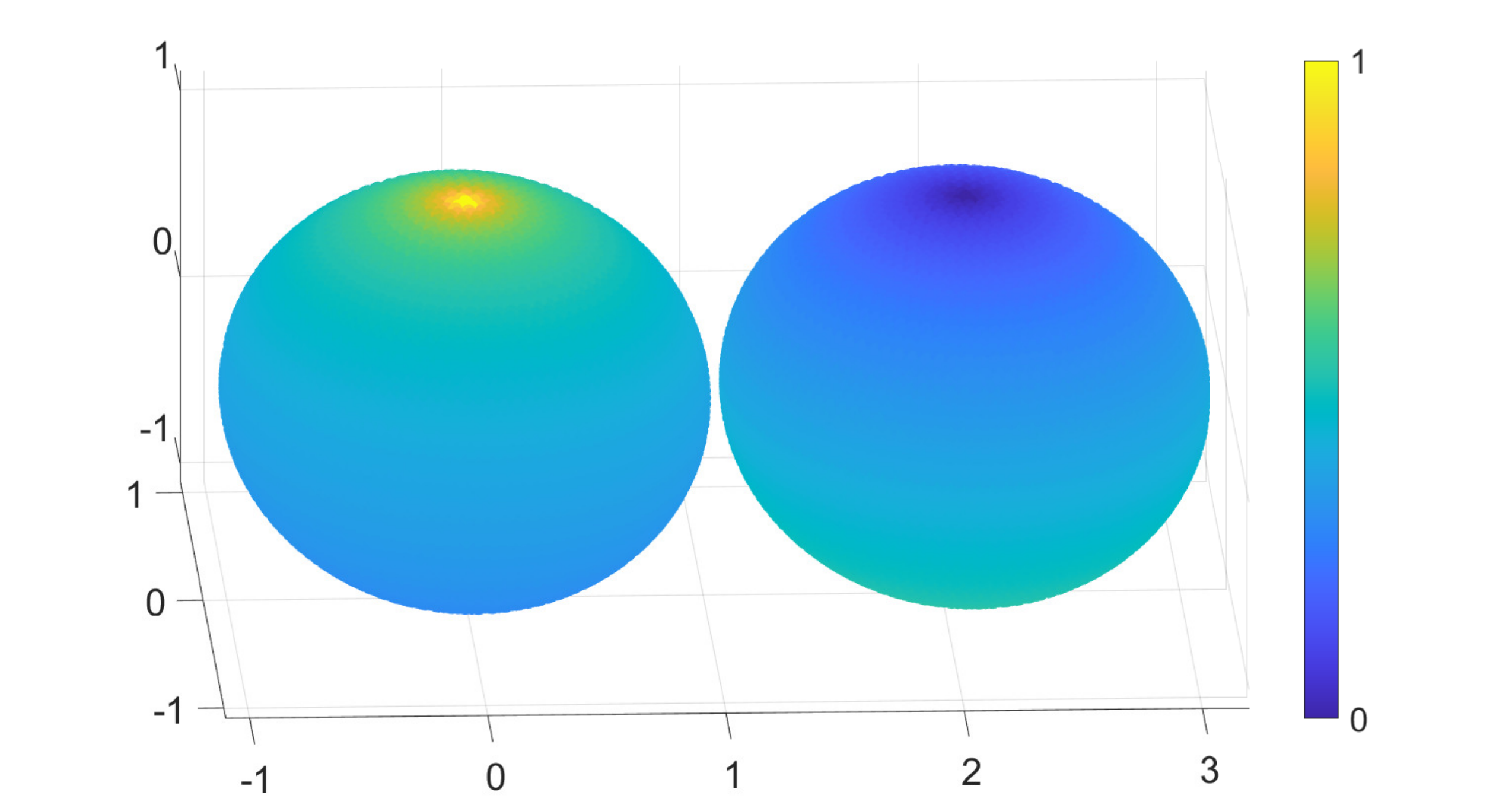}\label{fig:nonsmoothData}}
	\subfigure[]{\includegraphics[width=0.8\textwidth]{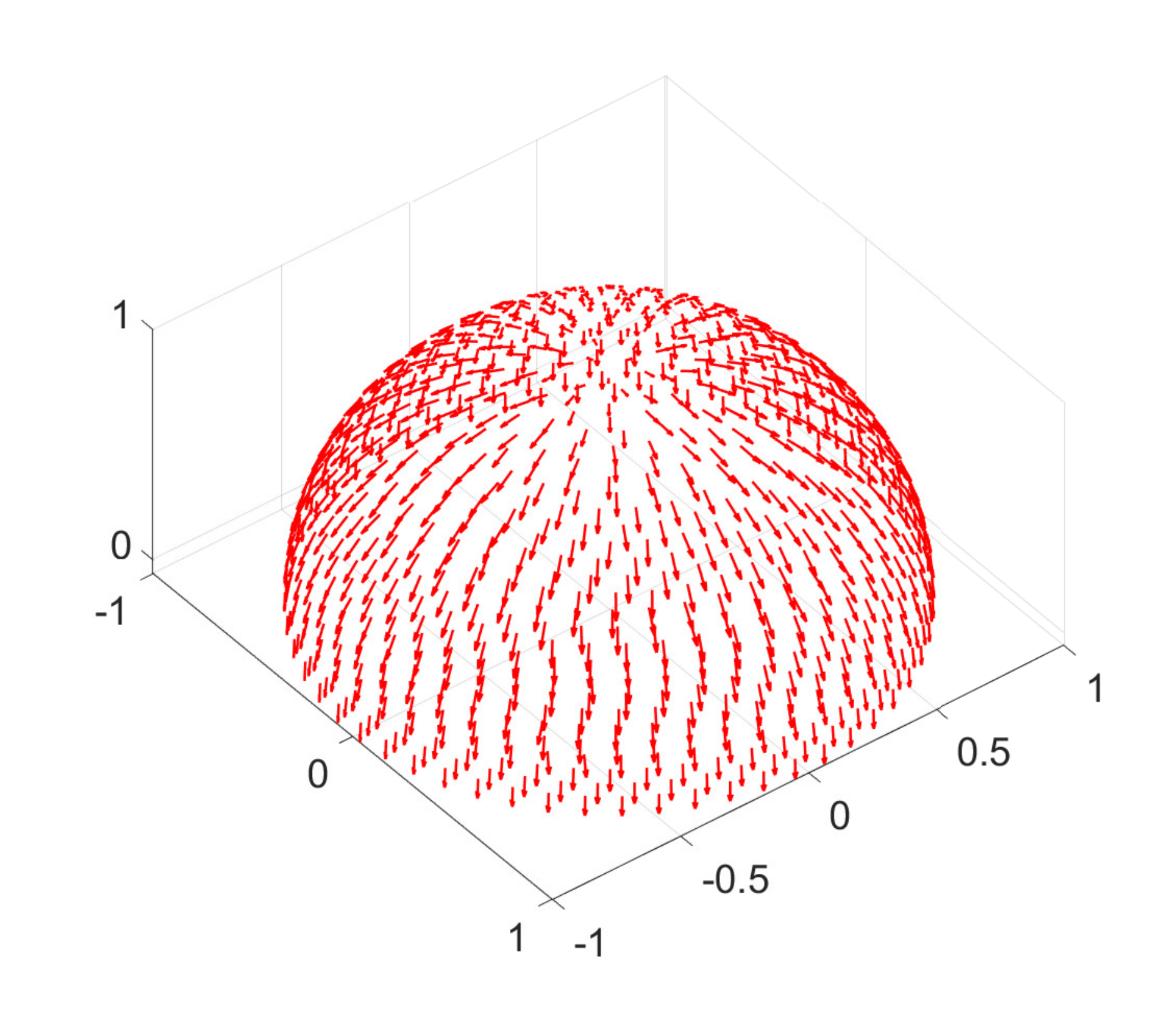}\label{fig:nonsmoothGrad}}
	\caption{\subref{fig:nonsmoothData}~Unbounded source and non-Lipschitz target densities and \subref{fig:nonsmoothGrad}~resulting gradient (visualized for the northern hemisphere) computed on a  $N=7722$ point layered grid. }
	\label{fig:nonsmooth}
\end{figure}

\section{Conclusion}\label{sec:conclusion}

We have introduced a provably convergent numerical method for solving the Optimal Transport problem on the sphere.  The method applies to both the squared geodesic cost function (important in mesh generation) and the logarithmic cost function (arising in the reflector antenna problem).  This method involves replacing the PDE on the manifold with equivalent PDEs on local tangent planes using a careful choice of geodesic normal coordinates.  The resulting PDE was approximated using monotone schemes inspired by recent methods for solving the \MA equation.  However, significant new techniques were introduced to handle the complicated nonlinear combination of gradient and Hessian terms.  Moreover, we introduced a smooth cutoff for the logarithmic cost function to enable well-posed, consistent, and monotone discretizations even in the presence of this potentially unbounded cost function.

We produced an implementation of this method and demonstrated via computations how our numerical method is able to address issues heretofore unresolved.  Notably, the method can handle both structured and unstructured grids, non-smooth data, and produces compelling results in examples closely related to both the moving mesh and reflector antenna design problems.

\bibliographystyle{plain}
\bibliography{OTonSphere2}

{\appendix
\section{Derivation of the Mixed Hessian}\label{app:mixedHessian}
In this appendix, we fill in the details of the derivation of simple expressions for the determinant of the mixed Hessian.  For each cost function, we take the following approach:
\begin{enumerate}
\item Introduce orthogonal perturbations $\Delta p_1, \Delta p_2 \in \Tf_x$ such that $\Delta p_1\cdot p = 0$ and $\Delta p_2 = \hat{p} \|\Delta p_2\|$.
\item Establish that $T(x,p)-T(x,p+\Delta p_1)$ and $T(x,p)-T(x,p+\Delta p_2)$ are orthogonal to leading order.
\item Compute the change of area formula
\begin{align*}
 \left|\det\right.&\left.(D_pT(x,p))\right|  \\ &=\lim_{\left\Vert \Delta p_1 \right\Vert, \left\Vert \Delta p_2 \right\Vert \rightarrow 0} \frac{d_{\mathbb{S}^2} \left( T(x,p), T(x,p+ \Delta p_1) \right) d_{\mathbb{S}^2} \left( T(x,p), T(x,p + \Delta p_2) \right) }{\left\Vert \Delta p_1 \right\Vert \left\Vert \Delta p_2 \right\Vert}\\
&=\lim_{\left\Vert \Delta p_1 \right\Vert, \left\Vert \Delta p_2 \right\Vert \rightarrow 0} \frac{\| T(x,p)-T(x,p+\Delta p_1)\| \| T(x,p)-T(x,p+\Delta p_1)\| }{\left\Vert \Delta p_1 \right\Vert \left\Vert \Delta p_2 \right\Vert}
\end{align*}
where we can simplify the formulas by using the fact that
\[\left\Vert T(x,p) - T(x,p+ \Delta p)\right\Vert = d_{\mathbb{S}^2}\left( T(x,p),T(x,p+\Delta p) \right) + \mathcal{O}\left(\left\Vert \Delta p \right\Vert^2 \right).\]
\end{enumerate}

\subsection{Squared geodesic cost}
We begin with the squared geodesic cost, recalling that the mapping $T$ has the explicit form
\[ T(x,p) = \cos\left(\norm{ p } \right) x + \sin\left(\norm{ p } \right)\frac{p}{ \norm{ p }}. \]

First consider a perturbation satisfying $\Delta p_1 \cdot p  = 0$.
\begin{align*}
T(x, p) - T(x, p+ \Delta p_1) = &x \left( \cos \left\Vert p \right\Vert - \cos \left\Vert p + \Delta p_1 \right\Vert \right)\\ &+ \frac{p}{\left\Vert p \right\Vert} \sin \left\Vert p \right\Vert - 
\frac{p + \Delta p_1}{\left\Vert p + \Delta p_1 \right \Vert} \sin \left\Vert p + \Delta p_1 \right\Vert
\end{align*}

Now, since $p$ and $\Delta p_1$ are orthogonal, 
\[\left\Vert p + \Delta p_1 \right\Vert = \sqrt{\left\Vert p \right\Vert^2 + \left\Vert \Delta p_1 \right\Vert^2} = \|p\| + \bO\left(\| \Delta p_1\|^2\right).\]

Thus to leading order we obtain
\[ T(x, p) - T(x, p+ \Delta p_1) =    \Delta p_1  \frac{\sin \left\Vert p \right\Vert}{\left\Vert p \right\Vert} +\bO\left(\left\Vert \Delta p_1 \right\Vert^2\right).\]

 %Thus, $1/\left\Vert p + \Delta p_1\right\Vert = \frac{1}{\left\Vert p \right\Vert} (1 - \frac{\left\Vert \Delta p_1 \right\Vert^2}{2\left\Vert p \right\Vert^2}) = 1/\left\Vert p \right\Vert + o(\left\Vert \Delta p_1 \right\Vert)$. Also,
%
%\begin{equation}
%\cos \left\Vert p + \Delta p_1 \right\Vert = \cos\left\Vert p \right\Vert + o(\left\Vert \Delta p_1 \right\Vert)
%\end{equation}
%and
%\begin{equation}
%\sin \left\Vert p + \Delta p_1 \right\Vert = \sin\left\Vert p \right\Vert + o(\left\Vert \Delta p_1 \right\Vert)
%\end{equation}
%
%Thus, we get:
%
%\begin{equation}
%\left\Vert T(x,p) - T(x, p + \Delta p_1) \right\Vert = \left\Vert \Delta p_1 \right\Vert \frac{\sin \left\Vert p \right\Vert}{\left\Vert p \right\Vert} + o(\left\Vert \Delta p_1 \right\Vert)
%\end{equation}

Now, suppose that $\Delta p_2 = \left\Vert \Delta p_2 \right\Vert \frac{p}{\left\Vert p \right\Vert}$. In this case, $\left\Vert p + \Delta p_2 \right\Vert = \left\Vert p \right\Vert + \left\Vert \Delta p_2 \right\Vert$.  As before, we compute to leading order:
\begin{align*}
T(x, p) &- T(x, p+ \Delta p_2) \\&= x \left( \cos \left\Vert p \right\Vert - \cos \left\Vert p + \Delta p_2 \right\Vert \right)+ \frac{p}{\left\Vert p \right\Vert} \sin \left\Vert p \right\Vert - 
\frac{p + \Delta p_2}{\left\Vert p + \Delta p_2 \right \Vert} \sin \left\Vert p + \Delta p_2 \right\Vert\\
  &= x \left\Vert \Delta p_2 \right\Vert \sin \left\Vert p \right\Vert - \frac{p}{\left\Vert p \right\Vert} \left\Vert \Delta p_2 \right\Vert \cos \left\Vert p \right\Vert + \frac{p}{\left\Vert p \right\Vert^2} \left\Vert \Delta p_2 \right\Vert \sin \left\Vert p \right\Vert - \frac{\Delta p_2}{\left\Vert p \right\Vert} \sin\left\Vert p \right\Vert\\
	&\phantom{=}+ \bO\left(\left\Vert \Delta p_2 \right\Vert^2\right)\\
	&= x \left\Vert \Delta p_2 \right\Vert \sin \left\Vert p \right\Vert - \frac{p}{\left\Vert p \right\Vert} \left\Vert \Delta p_2 \right\Vert \cos \left\Vert p \right\Vert + \bO\left(\left\Vert \Delta p_2 \right\Vert^2\right).
\end{align*}

Now since $x\cdot\Delta p_1 = p\cdot \Delta p_1 = \Delta p_1\cdot \Delta p_2 = 0$, we can easily verify that
\[ (T(x, p) - T(x, p+ \Delta p_1)) \cdot (T(x, p) - T(x, p+ \Delta p_2)) = o\left(\|\Delta p_1\| + \|\Delta p_2\|\right) \]
so that the perturbations in the map are indeed orthogonal to leading order.

Now we can use orthogonality to easily compute the magnitudes of these perturbations via
\[ \|T(x, p) - T(x, p+ \Delta p_1)\|^2 =  \|\Delta p_1\|^2  \frac{\sin^2 \left\Vert p \right\Vert}{\left\Vert p \right\Vert^2} +o\left(\left\Vert \Delta p_1 \right\Vert^2\right) \]
and
\begin{align*}
\|T(x, p) - T(x, p+ \Delta p_2)\|^2 &= \|\Delta p_2\|^2 \sin^2\|p\| + \|\Delta p_2\|^2 \cos^2\|p\| + o\left(\|\Delta p_2\|^2\right)\\
  &=  \|\Delta p_2\|^2 + o\left( \|\Delta p_2\|^2\right).
\end{align*}
where we have used the fact that $x$ and $p/\|p\|$ are unit vectors.

Then we can compute the change of area formula as
\begin{align*}
\left|\det\right.&\left.(D_pT(x,p))\right|  \\ &=\lim_{\left\Vert \Delta p_1 \right\Vert, \left\Vert \Delta p_2 \right\Vert \rightarrow 0} \frac{\| T(x,p)-T(x,p+\Delta p_1)\| \| T(x,p)-T(x,p+\Delta p_2)\| }{\left\Vert \Delta p_1 \right\Vert \left\Vert \Delta p_2 \right\Vert}\\
  &= \lim_{\left\Vert \Delta p_1 \right\Vert, \left\Vert \Delta p_2 \right\Vert \rightarrow 0}  \frac{\|\Delta p_1\| \|\Delta p_2\| \sin\|p\| / \|p\| + o\left(\| \Delta p_1\| \|\Delta p_2\|\right)}{\left\Vert \Delta p_1 \right\Vert \left\Vert \Delta p_2 \right\Vert}\\
	&= \frac{\sin\|p\|}{\|p\|}.
\end{align*}

Hence, the determinant of the mixed Hessian for the squared geodesic cost is
\begin{equation}
\left\vert \det D^2_{xy} c(x,y) \right\vert = \frac{\left\Vert p \right\Vert}{\sin \left\Vert p \right\Vert}.
\end{equation}

\subsection{Logarithmic cost}
Now we perform the same procedure for the logarithmic map, which has the explicit form
\[ T(x,p) = x\frac{\norm{ p }^2-1/4}{\norm{ p }^2+1/4}-\frac{p}{\norm{ p }^2+1/4}.\]

We again begin with a perturbation satisfying $p\cdot\Delta p_1 = 0$ so that
\[\left\Vert p + \Delta p_1 \right\Vert = \|p\| + \bO\left(\| \Delta p_1\|^2\right).\]

To leading order, we can compute
\begin{align*}
T(x,p)-T(x,p+\Delta p_1) &= x\frac{\norm{ p }^2-1/4}{\norm{ p }^2+1/4}- x\frac{\norm{ p + \Delta p_1 }^2-1/4}{\norm{ p + \Delta p_1 }^2+1/4}\\
  &\phantom{=}-\frac{p}{\norm{ p }^2+1/4} + \frac{p+\Delta p_1}{\norm{ p + \Delta p_1 }^2+1/4}\\
	&= \frac{\Delta p_1}{\|p^2\|+1/4} + \bO\left(\| \Delta p_1\|^2\right).
\end{align*}

Next we consider an orthogonal perturbation $\Delta p_2 = \left\Vert \Delta p_2 \right\Vert \frac{p}{\left\Vert p \right\Vert}$ so that \[\|p+\Delta p_2\|^2 = \|p\|^2 + 2\|p\|\|\Delta p_2\| + \bO\left(\| \Delta p_2\|^2\right).\]
Now we can compute to leading order
\begin{align*}
\|T(x, p) &- T(x,p+\Delta p_2)\| = x\frac{\norm{ p }^2-1/4}{\norm{ p }^2+1/4}-x\frac{\norm{ p }^2+2\|p\| \|\Delta p_2\|-1/4}{\norm{ p }+2\|p\| \|\Delta p_2\|^2+1/4}\\
 &\phantom{=} - \frac{p}{\norm{ p }^2+1/4} + \frac{p + \Delta p_2}{\norm{ p }^2 + 2\|p\| \|\Delta p_2\|+1/4} + \bO\left(\| \Delta p_2\|^2\right)\\
 &= -x\frac{2\|p\| \|\Delta p_2\|}{\|p\|^2 + 1/4} + x\frac{2\|p\| \|\Delta p_2\|(\|p\|^2-1/4)}{(\|p\|^2+1/4)^2}\\
 &\phantom{=} + \frac{\Delta p_2}{\|p\|^2+1/4} - p\frac{2\|p\| \|\Delta p_2\|}{(\|p\|^2+1/4)^2}+ \bO\left(\| \Delta p_2\|^2\right)\\
 &= -x\frac{\|p\| \|\Delta p_2\|}{(\|p\|^2+1/4)^2} + \hat{p}\frac{\|\Delta p_2\|(1/4-\|p\|^2)}{(\|p\|^2+1/4)^2} + \bO\left(\| \Delta p_2\|^2\right).
\end{align*}

Since $x$, $p$, and $\Delta p_1$ are mutually orthogonal, we can immediately verify that 
\[ (T(x, p) - T(x, p+ \Delta p_1)) \cdot (T(x, p) - T(x, p+ \Delta p_2)) = o\left(\|\Delta p_1\| + \|\Delta p_2\|\right) \]
so that the perturbations in the mapping are again orthogonal to leading order.

Next we compute the lengths of the perturbations using orthogonality:
\[
\left\Vert T(x, p) - T(x,p+\Delta p_1) \right\Vert = \left\Vert \Delta p_1 \right\Vert \frac{1}{\left\Vert p \right\Vert^2 + 1/4} + o(\left\Vert \Delta p_1 \right\Vert)
\]
and
\begin{align*}
\|T(x,p)-T(x,p+\Delta p_2)\|^2 &= \frac{\|p\|^2\|\Delta p_2\|^2 + \|\Delta p_2\|^2(1/4-\|p\|^2)^2}{(\|p\|^2+1/4)^4}+ o(\left\Vert \Delta p_2 \right\Vert^2)\\
  &= \frac{\|\Delta p_2\|^2}{(\|p\|^2+1/4)^4}\left(\|p\|^4 + \frac{1}{2}\|p\|^2+\frac{1}{16}\right)+ o(\left\Vert \Delta p_2 \right\Vert^2)\\
	&= \frac{\|\Delta p_2\|^2}{(\|p\|^2+1/4)^2} + o(\left\Vert \Delta p_2 \right\Vert^2).
\end{align*}

Then we can again compute the change of area formula as
\begin{align*}
\left|\det\right.&\left.(D_pT(x,p))\right|  \\ &=\lim_{\left\Vert \Delta p_1 \right\Vert, \left\Vert \Delta p_2 \right\Vert \rightarrow 0} \frac{\| T(x,p)-T(x,p+\Delta p_1)\| \| T(x,p)-T(x,p+\Delta p_2)\| }{\left\Vert \Delta p_1 \right\Vert \left\Vert \Delta p_2 \right\Vert}\\
  &= \lim_{\left\Vert \Delta p_1 \right\Vert, \left\Vert \Delta p_2 \right\Vert \rightarrow 0}  \frac{\|\Delta p_1\| \|\Delta p_2\| / (\|p\|^2+1/4)^2 + o\left(\| \Delta p_1\| \|\Delta p_2\|\right)}{\left\Vert \Delta p_1 \right\Vert \left\Vert \Delta p_2 \right\Vert}\\
	&= \frac{1}{(\|p\|^2 + 1/4)^2}.
\end{align*}

Hence, the determinant of the mixed Hessian for the logarithmic cost is
\begin{equation}
\left\vert \det D^2_{xy} c(x,y) \right\vert = \left( \left\Vert p \right\Vert^2 + 1/4 \right)^2.
\end{equation}

}

\end{document}